\documentclass[11pt,reqno]{amsart}
\usepackage{amsmath, amsthm, amssymb}
\usepackage{mathrsfs}
\usepackage{array}
\usepackage{caption}

\evensidemargin \oddsidemargin
\newcommand{\mmod}[1]{\,\,(\text{\rm mod}\,\, #1)}
\def\bfa{{\boldsymbol a}}

\def\bfn{{\boldsymbol n}}

\newtheorem{thm}{Theorem}
\newtheorem{cor}{Corollary}

\newtheorem{lem}{Lemma}

\newtheorem{prop}{Proposition}

\numberwithin{equation}{section} \numberwithin{thm}{section}
\numberwithin{lem}{section} \numberwithin{problem}{section}
\numberwithin{cor}{section}

\begin{document}
\title{Estimates for a three-dimensional exponential sum with monomials}
\author[Javier Pliego]{Javier Pliego}
\address{Department of Mathematics, KTH Royal Institute of Technology, Lindstedtsv\"agen 25,
10044 Stockholm, Sweden}

\email{javierpg@kth.se}

\begin{abstract} We examine a family of three-dimensional exponential sums with monomials and provide estimates which are in some instances sharper than those stemming from approaches entailing the use of existing bounds pertaining to analogous sums.
\end{abstract}
\maketitle

\section{Introduction}
Exponential sums with the corresponding phase being a smooth function make their appearance on innumerous instances in the analytic theory of numbers, it often being the case that progress on many problems in the field hinges on robust enough estimates for such sums. The first succesful methods to derive non-trivial bounds for one-dimensional sums
$$\sum_{n}e(f(n))$$ were developed by Weyl \cite{Wey}, van der Corput \cite{Van} and Vinogradov \cite{Vin}, these developments having applications on the question of computing the order of $\zeta(\sigma+it)$ for fixed $1/2\leq \sigma\leq 1$ and on various divisor problems.  A reappraisal of the ideas of van der Corput enabled Philipps \cite{Phi} to initiate the theory of exponent pairs, the employment of such a technique in the aforementioned problems ultimately delivering sharper conclusions. The same avenue led to Rankin \cite{Ran} and Graham \cite{Grako} by essentially optimizing similar arguments accordingly to improve the previous bound on the order of $\zeta(1/2+it)$.

We find it opportune to note upon considering a suitable domain $\mathcal{D}\subset\mathbb{R}^{2}$ and a function $f$ satisfying some smoothness conditions that estimates for two dimensional sums of the form
$$\sum_{(m,n)\in \mathcal{D}\cap \mathbb{N}^{2}}e\big(f(m,n)\big)$$ which don't just make a trivial use of the observation 
$$\Big\lvert \sum_{(m,n)\in \mathcal{D}\cap \mathbb{N}^{2}}e\big(f(m,n)\big)\Big\rvert\leq \sum_{m}\Big\lvert \sum_{(m,n)\in \mathcal{D}\cap \mathbb{N}^{2}}e\big(f(m,n)\big)\Big\rvert,$$ in conjunction with an application of available estimates for one-dimensional sums, often lead to better conclusions in problems within this circle of ideas. An analogous theory of exponent pairs for two dimensional sums was then developed and led Kolesnik \cite{Kol} to make progress on the above problems.

In particular, multidimensional exponential sums $$\sum_{n_{1},\ldots,n_{s}}e\big(f(n_{1},\ldots,n_{s})\big)$$ for the choice of monomials $f(n_{1},\ldots,n_{s})=xn_{1}^{\alpha_{1}}\dots n_{s}^{\alpha_{s}}$ have received significant attention due to their relevance in various problems in the field, it being pertinent to highlight the estimates for a collection of families of such sums obtained by Fouvry and Iwaniec \cite{Fouv} by means of a novel application of the double large sieve in conjunction with a corresponding spacing lemma. The work of Robert and Sargos \cite{Rob} delivered by sharpening the aforementioned spacing lemma a refinement of the result in the preceding paper for sums of the shape
\begin{equation}\label{may}\sum_{h\asymp H}\sum_{n\asymp N}\phi(h,n)\sum_{m\asymp M}\lambda(m)e\big(Xh^{b}n^{c}m^{a}\big),\end{equation} wherein $H,N,M\in\mathbb{N}$, the parameter $X\in\mathbb{R}$ has the property that $XH^{a}N^{b}M^{c}>1$, the weights satisfy $\lvert \phi(h,n)\rvert,\lvert \lambda(m)\rvert\leq 1$ and $a, b,c\in\mathbb{R}$ are fixed and $a(a-1)bc\neq 0.$

We shall focus our attention in this memoir on a family of exponential sums for particular choices of weights and range of summation, its origin and motivation being detailed at a later point in the introduction.

\begin{thm}\label{thm2.3}
Let $a,b,c\in\mathbb{R}$ such that $0<a<c< b$ and having the property that $b+c-a=1$ and $c<2a$. Then for big enough $T$ one has that
\begin{equation}\label{thesum}\sum_{(h,n,m)\in\mathcal{D}_{T}}\omega(h,n,m)e\big(\kappa h^{b}n^{c}m^{-a}\big)\ll T^{5/4-c/4a}+T^{1/4+(2a-c)/2(b-c)},\end{equation} wherein $\omega(h,n,m)=h^{-1/2+b/2}n^{-1/2+c/2}m^{-1/2-a/2}$, $\kappa$ is a constant that shall be introduced in (\ref{cncn})  and the domain $\mathcal{D}_{T}$ is defined by 
$$\mathcal{D}_{T}=\Big\{(h,n,m)\in\mathbb{N}^{3}:\ \ \ (hnm)^{2/3}(abc)^{-1/3}\leq h^{b}n^{c}m^{-a}\leq T\Big\}.$$
\end{thm}

We refer the reader to the Appendix \ref{pssd} for a detailed exposition of alternative arguments to bound the above exponential sum by making use of exponent pairs and on estimates for (\ref{may}) available in the literature. We nonetheless content ourselves to mention that the result in the previous theorem is stronger for the range \begin{equation}\label{range}\frac{42a+34c}{55}<b<2a+c\end{equation} than the corresponding estimates which would stem from the aforementioned arguments, the strongest of such being of the shape $O(T^{3/4+(2a-c)/2(b-c)}).$

Estimates for sums of the form (\ref{may}) have applications inter alia in the abelian group problem (see \cite{Liu,Sar,Rob}), the problem of $\mathcal{B}$-free numbers in short intervals (see \cite{Wu1,Sar}) or the distribution of $4$-full numbers (see \cite{Liu2,Sar}). The fact that the estimate obtained herein appertains the region $\mathcal{D}_{T}$ though precludes one from deriving similar estimates over dyadic intervals for the purpose of both eliminating the corresponding weight in the sum and exploring potential applications to problems in the vein of those earlier mentioned.

In contrast with previous work for bounding exponential sums with monomials we shall utilise herein a different approach and examine for a parameter $T>0$ and fixed $a,b,c>0$ satisfying $b+c-a=1$ the integral
$$I_{a,b,c}(T)=\int_{0}^{T}\zeta(1/2+ia t)\zeta(1/2-ib t)\zeta(1/2-ic t)dt$$ and analyse it in two different manners. We first observe that an application of the approximate functional equation (see Titchmarsh \cite[(4.12.4)]{Tit}) to each of the zeta functions reduces the task to that of computing eight integrals of products of Dirichlet polynomials, as was performed in earlier work of the author (see \cite[(7.5)]{Pli}), and thus delivers
\begin{equation}\label{sgmi}I_{a,b,c}(T)=\sigma_{a,b,c}T+M_{1}(T)+J_{2,2}(T)+O\big(T^{3/4}(\log T)+T^{1/2+a/2c}(\log T)^{2}+T^{5/4-c/4a}\big),\end{equation} wherein $M_{1}(T)$ and $J_{2,2}(T)$ shall be defined in (\ref{M1T}) and (\ref{integr}) respectively.

The main contribution of this paper lies on the assessment of establishing and employing an approximate functional equation for the product $$\zeta(1/2+ia t)\zeta(1/2-ib t)\zeta(1/2-ic t),$$ the procedure utilised to such an end being inspired by that of Heath-Brown \cite{Heath}. The initial stages of this endeavour contain as in \cite{Heath} the use of both the functional equation for the Riemann zeta function and Cauchy's residue theorem in conjunction with succesive applications of Stirling's formula. The resulting equation then comprises a first term that shall contribute to the main term in (\ref{sgmi}) after integrating over $t$ and a second term from which the sum in the left side of (\ref{thesum}) arises after an application of the stationary phase method. The main term pertaining to the off-diagonal contribution which we denote by $I_{1,2}(T)$ then satisfies
$$\lvert M_{1}(T)+J_{2,2}(T)-I_{1,2}(T)\rvert\ll T^{5/4-c/4a},$$
the term in the right side of the above equation being smaller than the bounds which may be conditionally obtained under the assumption of the $abc$ conjecture for each of the individual summands on the left side (see \cite[Lemmata 3.2, 6.3]{Pli}). The argument is culminated by comparing the two formulae obtained for $I_{a,b,c}(T)$ for the purpose of deriving an estimate for (\ref{thesum}), it being noteworthy that such an approach yields for some ranges stronger bounds than those delivered by the use of exponent pairs.

There are some additional terms which emerge in the analysis, the treatment of which departing from that of Heath-Brown \cite{Heath} in that we bound those by means of appropiate oscillatory integral estimates. The argument in \cite{Heath} instead entails integrating by parts analogous products, such an approach in our context being insufficient and diverting one to the undesirable position of encountering sums containing the factor $$\log\Big(\frac{n_{2}^{b}n_{3}^{c}}{n_{1}^{a}}\Big)^{-1},$$ for which we have a poor understanding. The analysis of the contribution stemming from the main term in the approximate functional equation in \cite{Heath} makes an elegant use of the underlying symmetry to exhibit further cancellation when integrating such a term twisted by $$\Big(\frac{m}{n}\Big)^{it}.$$ In the absence of such a property herein, our analysis comprises a careful examination of the corresponding phases that will eventually lead to a division of the corresponding tuples depending on the size of the phases, such an intrincate process being ultimately culminated with a routinary application of oscillatory integral estimates.

The structure of the paper is organised as follows: Section \ref{sec6.2} is devoted to a prolix discussion concerning the approximate functional equation. The diagonal and off-diagonal contribution are computed in Sections \ref{sec6.4} and \ref{sec6.5} respectively. The examination of various residual terms utilising many of the ideas from its preceding section is performed in Section \ref{sec6.6}. Section \ref{sec6.7} comprises the application of the stationary phase method from which the exponential sum arises. An appendix discussing other possible approaches for estimating the sum in (\ref{thesum}) is included at the end of the memoir, the bounds derived therein being weaker to those stemming from our main theorem.

\emph{Notation}: We write $[x]$ for $x\in\mathbb{R}$ to denote the nearest integer and $$||x||=x-[x].$$ Whenever $\varepsilon$ appears in any bound, it will mean that the bound holds for every $\varepsilon>0$, though the implicit constant then may depend on $\varepsilon$. We use $\ll$ and $\gg$ to denote Vinogradov's notation. When we employ such a notation to describe the limits of summation of a particular sum we shall only be interested in estimating such a sum, and the precise value of the implicit constant won't have any impact in the argument.

\emph{Acknowledgements}: The author's work was initiated during his visit to Purdue University under Trevor Wooley's supervision and finished at KTH while being supported by the G\"oran Gustafsson Foundation. The author would like to thank him for his guidance and helpful comments and both Purdue University and KTH for their support and hospitality.

\section{The approximate functional equation}\label{sec6.2}

We begin by furnishing ourselves with a lemma which essentially follows \cite{Heath} and shall ultimately provide the approximate functional equation to which we alluded in the introduction. It has been thought preferable to present the lemma in wider generality, it being convenient to such an end to define for each natural number $k\geq 2$ the subset $\mathcal{A}_{k}\subset(\mathbb{R}\setminus\{0\})^{k}$ of tuples $\bfa=(a_{1},\ldots,a_{k})$ satisfying the inequalities
\begin{equation}\label{AA}a_{j}^{2}>\frac{\pi}{4}\Big(-\xi_{\bfa}a_{j}+\sum_{l=1}^{k}\lvert a_{l}\rvert-\sum_{l=1}^{k}\lvert a_{l}-a_{j}\rvert\Big),\ \ \ \ \ \ \ \ (1\leq j\leq k)\end{equation} where in the above line the number $\xi_{\bfa}$ is defined by means of the formula \begin{equation}\label{xi}\xi_{\bfa}=\sum_{a_{l}>0}1-\sum_{a_{l}<0}1.\end{equation}It seems adecquate to introduce for tuples $\bfn=(n_{1},n_{2},\ldots,n_{k})$ and $\bfa$ as above the parameters 
\begin{equation}\label{ISP}I_{\bfa}=\sum_{l=1}^{k}\frac{1}{a_{l}},\ \ \ \ Q_{\bfa}=\prod_{j=1}^{k}a_{j},\ \ \ \ P_{\bfn}=\prod_{j=1}^{k}n_{j},\ \ \ \ L_{\bfa}(\bfn)=\prod_{j=1}^{k}n_{j}^{-a_{j}}.\end{equation}
We find it desirable to consider the product of gamma functions
$$P_{\bfa}(t)=\prod_{j=1}^{k}\Gamma\big(1/2(1/2+ia_{j}t)\big),$$ which shall make its appearance in the course of the discussion concerning the approximate functional equation. Likewise, we further define
\begin{equation}\label{Gmi}G_{m}(z,t)=P_{\bfa}(t)^{-1}\prod_{j=1}^{k}\Gamma\Big(\frac{1}{2}\Big(\frac{1}{2}+(-1)^{m+1}ia_{j}t+z\Big)\Big),\ \ \ \ \ \ m=1,2.\end{equation}
It may also seem appropiate to recall (\ref{xi}) and introduce the smoothing factor 
\begin{equation}\label{Htz}H(z,t)=e^{z^{2}/t-i\xi_{\bfa}\pi z/4}.\end{equation}

\begin{lem}\label{lemHB}
Let $\bfa=(a_{1},\ldots,a_{k})\in\mathcal{A}_{k}$. Then for $t>1$ one has that
\begin{equation}\label{zet}\prod_{j=1}^{k}\zeta(1/2+ia_{j}t)=\sum_{\bfn\in\mathbb{N}^{k}}P_{\bfn}^{-1/2}I(P_{\bfn},t)+O(t^{-2}),\end{equation}
wherein the function $I(P_{\bfn},t)$ at hand is defined by means of the equation $$I(P_{\bfn},t)=L_{\bfa}(\bfn)^{it}I_{1}(P_{\bfn},t)+L_{\bfa}(\bfn)^{-it}I_{2}(P_{\bfn},t),$$ the terms $I_{1}(x,t)$ and $I_{2}(x,t)$ for $x\in\mathbb{R}$ being $$I_{m}(x,t)=\frac{1}{2\pi i}\int_{1-i\infty}^{1+i\infty}G_{m}(z,t)\big(\pi^{k/2}x)^{-z}H\big((-1)^{m+1}z,t\big)\frac{dz}{z},\ \ \ \ \ \ m=1,2.$$ 
\end{lem}
\begin{proof}

We define, for convenience, the meromorphic function
$$f_{\bfa}(w)=\pi^{-kw/2}\prod_{j=1}^{k}\Gamma\big(1/2(w+ia_{j}t)\big)\zeta(w+ia_{j}t)$$ with poles at $w=-ia_{j}t$ and $w=1-ia_{j}t$ in the region $\text{Re}(w)\geq -3/2$. For ease of notation it has been thought preferable to denote henceforth $f_{\bfa}(w)$ by $f(w)$. We also consider 
$$B_{m}(t)=\frac{1}{2\pi i}\int_{(-1)^{m+1}-i\infty}^{(-1)^{m+1}+i\infty}f(1/2+z)H(z,t)\frac{dz}{z},\ \ \ \ \ m=1,2.$$  The reader may find it useful to observe that when $s\in\mathbb{R}$ then an application of the functional equation for the Riemann zeta function yields
$$f(-1/2+is)=\pi^{-k(3/2-is)/2}\prod_{j=1}^{k}\Gamma\big(1/2(3/2-is-ia_{j}t)\big)\zeta(3/2-is-ia_{j}t),$$ whence on defining
$$Y(z,t)=\pi^{-k/4}\pi^{-kz/2}\prod_{j=1}^{k}\Gamma\big(1/2(1/2-ia_{j}t+z)\big)\zeta(1/2-ia_{j}t+z)$$ and making a change of variables accordingly, it transpires that
\begin{equation*}B_{2}(t)=-\frac{1}{2\pi i}\int_{1-i\infty}^{1+i\infty}Y(z,t)H(-z,t)\frac{dz}{z}.\end{equation*}

It then seems pertinent to note that one may utilise the convergence of the series cognate to the Riemann zeta function at $\text{Re}(z)=3/2$ in conjunction with (\ref{Gmi}) to deduce
\begin{equation}\label{M111}B_{m}(t)=(-1)^{m+1}\pi^{-k/4}P_{\bfa}(t)\sum_{\bfn\in\mathbb{N}^{3}}P_{\bfn}^{-1/2}L_{\bfa}(\bfn)^{(-1)^{m+1}it}I_{m}(P_{\bfn},t),\ \ \ \ \ \ \ m=1,2.\end{equation}
It therefore transpires that an application of Cauchy's residue theorem already delivers the formula
$$B_{1}(t)-B_{2}(t)=f(1/2)+\frac{1}{2\pi i}\sum_{m=1}^{2k}\int_{\mathcal{C}_{m}}f(1/2+z)H(z,t)\frac{dz}{z},$$where in the above equation $\mathcal{C}_{m}$ denotes a circular path of radius $t^{-1}$ around each of the poles of $f(1/2+w).$ To the end of deriving the desired result one should then estimate the contribution of the remaining residues. We employ Stirling's series (see Whittaker and Watson \cite[\S 13.6]{Whi}), namely
\begin{equation}\label{gamma}\log \Gamma(z)=(z-1/2)\log z-z+\frac{1}{2}\log (2\pi)+ \sum_{r=1}^{N}c_{r}z^{1-2r}+O(\lvert z\rvert^{-1-2N}),\end{equation} where $c_{r}$ are fixed coefficients and $N\in\mathbb{N}$, for the purpose of observing that whenever the pole at hand pertaining to the function $f(1/2+w)$ is either $w=1/2-ia_{j}t$ or $w=-1/2-ia_{j}t$ for some fixed $j\leq k$ one may deduce the bound
$$f(1/2+w)\ll t^{C}e^{-C_{j}\pi t/4}$$ in the corresponding contour cognate to the  poles, with $C$ being a positive constant depending on the tuple $\bfa$ and $$C_{j}=\sum_{l\leq k}|a_{l}-a_{j}|.$$ Likewise, under the same circumstances the estimate $$H(z,t)\ll e^{-a_{j}^{2}t-\xi_{\bfa}\pi a_{j}t/4}$$ holds on the same contour, whence the preceding discussion then yields
\begin{equation}\label{J1213}f(1/2)=B_{1}(t)-B_{2}(t)+O(t^{C}\max_{j}e^{-a_{j}^{2}t-\xi_{\bfa}\pi a_{j}t/4-C_{j}\pi t/4}).\end{equation} 

It therefore remains to divide the above equation by $\pi^{-k/4}P_{\bfa}(t)$, a concomitant requisite being the estimation of the inverse of such a product, and this we achieve by means of a routine application of Stirling's formula. To this end, it might be worth considering $$\lambda_{\bfa}=\sum_{l=1}^{k}\lvert a_{l}\rvert,$$and noting that Stirling's formula then yields 
$$P_{\bfa}(t)^{-1}\ll  e^{\pi \lambda_{\bfa}t/4}.$$
We then divide both sides of (\ref{J1213}) by the product $P_{\bfa}(t)$ and combine it with equation (\ref{M111}) to the end of obtaining
\begin{align*}\label{zet}\prod_{j=1}^{k}\zeta(1/2+ia_{j}t)=\sum_{\bfn\in\mathbb{N}^{3}}P_{\bfn}^{-1/2}I\big(P_{\bfn},t\big)+E(t),\end{align*}where in the above equation the corresponding error term $E(t)$ satisfies $$E(t)\ll t^{C}\max_{j\leq k}e^{-a_{j}^{2}t-\xi_{\bfa}\pi a_{j}t/4+(\lambda_{\bfa}-C_{j})\pi t/4}.$$
We find it appropiate to remark that in view of the condition (\ref{AA}), it transpires that $$E(t)\ll e^{-Kt}$$ for some constant $K>0$. The preceding remark then in conjunction with the above formula delivers the desired result.

\end{proof}

For the purpose of progressing in the proof, it seems pertinent by means of a routine application of Stirling's formula to decompose the integrand involved in the expression for $I_{m}(P_{\bfn},t)$ into a main term and a secondary term. Before embarking ourselves in such an endeavour, it may be convenient to recall (\ref{ISP}) and define \begin{equation}\label{AAA}A(x,t)=Q_{\bfa}^{1/2}(t/2\pi)^{k/2}/x,\end{equation} and to remind the reader of the definition of $\mathcal{A}_{k}$ right above (\ref{AA}) and $I_{\bfa}$ in (\ref{ISP}).

\begin{lem}\label{lem0ita}
Let $k\geq 2$ and $\bfa=(a_{1},\ldots,a_{k})\in\mathcal{A}_{k}$. Then there exist constants $c_{i}(u,v)\in\mathbb{C}$ with $i=1,2$ for which for $t>1$ then
\begin{equation*}\prod_{j=1}^{k}\zeta(1/2+ia_{j}t)=\sum_{\bfn\in\mathbb{N}^{k}}P_{\bfn}^{-1/2}W(\bfn,t)+O(t^{-1}),\end{equation*} where in the above equation the function $ W(\bfn,t)$ is defined as
$$W(\bfn,t)=L_{\bfa}(\bfn)^{it}W_{1}(P_{\bfn},t)+L_{\bfa}(\bfn)^{-it}W_{2}(P_{\bfn},t),$$ the alluded terms $W_{1}(x,t)$ and $W_{2}(x,t)$ for $x\in\mathbb{R}_{+}$ being
$$W_{1}(x,t)=\frac{1}{2\pi i}\int_{1-i\infty}^{1+i\infty}A(x,t)^{z}F_{1}(z,t)\Big(1+\sum_{u,v}c_{1}(u,v)z^{u}t^{-v}\Big)\frac{dz}{z},$$ and
$$W_{2}(x,t)=\frac{\psi(t)}{2\pi i}\int_{1-i\infty}^{1+i\infty}A(x,t)^{z}F_{2}(z,t)\Big(1+\sum_{u,v}c_{2}(u,v)z^{u}t^{-v}\Big)\frac{dz}{z},$$where the functions $F_{m}(z,t)$ and $\psi(t)$ are defined as
\begin{equation}\label{F1F1}F_{1}(z,t)=e^{z^{2}/t-iI_{\bfa}z^{2}/4t},\ \ \ \ \ \ F_{2}(z,t)=e^{z^{2}/t+iI_{\bfa}z^{2}/4t},\ \ \ \ \ \ \psi(t)=e^{\xi_{\bfa}\pi i/4-ig_{\bfa}(t)},\end{equation} and the function $g_{\bfa}(t)$ is defined by means of the formula
\begin{equation}\label{ga}g_{\bfa}(t)=\sum_{j=1}^{k}a_{j}t\big(\log (\lvert a_{j}\rvert t/2)-1\big).\end{equation}Moreover, in the above sums the parameters $(u,v)$ run over the tuples satisfying $1\leq u\leq 3v/2$ and $1\leq v\leq 2(k+5)$ with the property that if $u\geq v+1$ then $v\geq 2$.
\end{lem}
\begin{proof}
We observe that in view of Lemma \ref{lemHB} it transpires that showing the validity of the above asymptotic evaluation amounts to proving for $x\in \mathbb{R}^{+}$ that 
\begin{equation}\label{ecI}I_{m}(x,t)-W_{m}(x,t)\ll x^{-1}t^{-1},\ \ \ \ \ \ m=1,2,\end{equation} since then the corresponding error term that arises when substituting $I_{m}(P_{\bfn},t)$ by $W_{m}(P_{\bfn},t)$ in (\ref{zet}) will be bounded above by
$$t^{-1}\sum_{\bfn\in\mathbb{N}^{3}}P_{\bfn}^{-3/2}\ll t^{-1}.$$

It may be worth analysing first $I_{1}(x,t)$. For such matters it seems convenient to introduce the parameters $\beta=z/2$, $\alpha_{j}=ia_{j}t/2$ and $\gamma_{j}=1/4+\alpha_{j}$ for each $j$. We shall henceforth assume that $\text{Re}(z)=1$ and confine ourselves first to the analysis of the function $G_{1}(z,t)$ when $\lvert \text{Im}(z)\rvert\leq t^{1/2}\log t.$ A routinary application of Stirling's formula (\ref{gamma}) with the choice $N=\lceil k/4+3/4\rceil$ delivers
\begin{align*}
\log \Gamma(\gamma_{j}+\beta)-\log \Gamma(\gamma_{j})=
&\beta\log\gamma_{j}+(\gamma_{j}+\beta-1/2)\log\big(1+\beta/\gamma_{j}\big)
\\
&-\beta+\sum_{r=1}^{N}c_{r}\big((\gamma_{j}+\beta)^{1-2r}-\gamma_{j}^{1-2r}\big)+O(t^{-1-2N}).
\end{align*}
The reader may notice that the use of the Taylor expansion of both $\log (1+w)$ and $(1+w)^{-1}$ reveals that the above equation equals
\begin{equation}\label{begabe}\log \Gamma(\gamma_{j}+\beta)-\log \Gamma(\gamma_{j})=\beta\log \gamma_{j}+\frac{\beta^{2}}{2\gamma_{j}}+\sum_{(u,v)}c_{1}'''(u,v)\beta^{u}\gamma_{j}^{-v}+O(t^{-5/2-k/2}),\end{equation}
 wherein the above sum $(u,v)$ runs over the tuples satisfying $1\leq u\leq v+1$ and $1\leq v\leq k+5$ with the property that if $u= v+1$ then $v\geq 2$, and $c_{1}'''(u,v)$ are fixed coefficients. The reader may observe that the rest of the terms stemming from the application of the Taylor expansion thereof may be absorbed into the error term therein. 

It also seems pertinent to note that another application of the Taylor expansions yields
$$\gamma_{j}^{-n}=\alpha_{j}^{-n}+\alpha_{j}^{-n}\sum_{m=1}^{\infty}k_{1,m}\alpha_{j}^{-m}\ \ \ \text{and}\ \ \ \log \gamma_{j}=\log\alpha_{j}+\sum_{m=1}^{\infty}k_{2,m}\alpha_{j}^{-m},$$wherein $k_{1,m},k_{2,m}\in\mathbb{R}$ denote as is customary fixed coefficients. These expressions in conjunction with the above equation and the definitions for $\beta$ and $\gamma_{j}$ earlier described then enable one to express the right side of (\ref{begabe}) as
\begin{equation*}\frac{z}{2}\big(\log(\lvert a_{j}\rvert t/2)+i\text{sgn}(a_{j})\pi/2\big)-i\frac{z^{2}}{4ta_{j}}+\sum_{u,v}c_{1}''(u,v)z^{u}\alpha_{j}^{-v}+O(t^{-5/2-k/2}),\end{equation*}
where herein $(u,v)$ runs over the collection of tuples earlier described and $c_{1}''(u,v)$ denote fixed real coefficients. Therefore, by recalling (\ref{Gmi}) and summing over $j$ we obtain
\begin{align*}\log G_{1}(z,t)=&
\frac{z}{2}\log\Big((t/2)^{k}\prod_{j=1}^{k}\lvert a_{j}\rvert\Big)+zi\pi\xi_{\bfa}/4-i\frac{I_{\bfa}}{4t}z^{2}+\sum_{u,v}c_{1}'(u,v)z^{u}(it)^{-v}
\\
&+O(t^{-5/2-k/2}),
\end{align*} whence raising the above equation to the power $e$, multiplying accordingly and employing (\ref{F1F1}) in the real line $\text{Re}(z)=1$ then yields
\begin{align}\label{G1H}
(\pi^{k/2}x)^{-z}G_{1}(z,t)H(z,t)z^{-1}=&
A(x,t)^{z}F_{1}(z,t)z^{-1}\Big(1+\sum_{u,v}c_{1}(u,v)z^{u}t^{-v}\Big)\nonumber
\\
&+O\Big(x^{-1}t^{-5/2}\lvert z\rvert^{-1}e^{\lvert I_{\bfa}\text{Im}(z)\rvert/2t -(\text{Im}(z))^{2}/t}\Big),
\end{align} wherein the reader may find it useful to recall the definition (\ref{Htz}) and where $(u,v)$ lies in the range described right after (\ref{ga}). We should note that both $c_{1}'(u,v)$ and $c_{1}(u,v)$ in the above equations denote fixed coefficients. By integrating the above equation over the segment [$1-it^{1/2}\log t, 1+it^{1/2}\log t]$, it transpires that the contribution $C_{1}(x,t)$ stemming from the error term will satisfy
\begin{equation}\label{Errori}C_{1}(x,t)\ll x^{-1}t^{-2}.\end{equation}

In order to complete the desired approximation, it seems pertinent to investigate the function $G_{1}(z,t)$ at hand whenever $\lvert \text{Im}(z)\rvert>t^{1/2}\log t$. For ease of notation we denote $y=\text{Im}(z)$, and apply (\ref{gamma}) on the range $\lvert y\rvert>t^{1/2}\log t$ to obtain
\begin{align*}
\log \Gamma(\gamma_{j}+\beta)-\log \Gamma(\gamma_{j})=&
(\gamma_{j}+\beta)\log(\gamma_{j}+\beta)-\gamma_{j}\log(\gamma_{j})
\\
&-\beta-1/2\log(1+\beta/\gamma_{j})+O(1),
\end{align*}
whence taking real parts in the above expression yields
\begin{align*}\log \lvert\Gamma(\gamma_{j}+\beta)\rvert-\log \lvert\Gamma(\gamma_{j})\rvert=&
-\frac{\pi}{4}(a_{j}t+y)\text{sgn}(a_{j}t+y)+\frac{\pi}{4}a_{j}t\cdot\text{sgn}(a_{j})+O\big(\log(t+\lvert y\rvert)\big).
\end{align*}
It might be worth noting that on recalling (\ref{Gmi}) it follows that
$$ \log\lvert G_{1}(z,t)\rvert= \sum_{j=1}^{k}\log \lvert\Gamma(\gamma_{j}+\beta)\rvert-\log \lvert\Gamma(\gamma_{j})\rvert,$$ whence in the interest of deriving an estimate of an appropiate precision it seems pertinent to show the inequality
\begin{equation}\label{inecua}-\sum_{j=1}^{k}\lvert a_{j}t+y\rvert+\sum_{j=1}^{k}\lvert a_{j}t\rvert+\xi_{\bfa} y\leq 0.\end{equation}The reader may observe that such a bound follows from the estimates 
$$-\sum_{a_{j}>0}\lvert a_{j}t+y\rvert+\sum_{a_{j}>0} a_{j}t+y\sum_{a_{j}>0}1\leq 0,\ \ \ \ \ \ \  -\sum_{a_{j}<0}\lvert a_{j}t+y\rvert-\sum_{a_{j}<0} a_{j}t-y\sum_{a_{j}<0}1\leq 0,$$ which in turn are an immediate consequence of the triangle inequality. Therefore, combining the previous bounds we find that $$\lvert G_{1}(z,t)\rvert\ll (yt)^{C}e^{-\pi\xi_{\bfa}y/4}$$ for some constant $C>0$. Such an estimate in conjunction with the definition (\ref{Htz}) yields
$$(\pi^{k/2}x)^{-z}G_{1}(z,t)H(z,t)z^{-1}\ll x^{-1}(yt)^{C}e^{-y^{2}/t}.$$Likewise,  by recalling (\ref{AAA}) one may deduce under the same circumstances that
$$A(x,t)^{z}F_{1}(z,t)z^{-1}\big(1+\sum_{u,v}c_{1}(u,v)z^{u}t^{-v}\big)\ll x^{-1}(yt)^{C}e^{\lvert yI_{\bfa}\rvert/(2t)-y^{2}/t}.$$ 
We integrate (\ref{G1H}) over the lines $\text{Re}(z)=1$ and $\lvert \text{Im}(z)\rvert>t^{1/2}\log t$ and utilise (\ref{Errori}) in conjunction with the above inequalities to obtain (\ref{ecI}) for the case $m=1$, as desired.

In order to make progress in our endeavour we shall next examine the term $I_{2}(x,t)$. We henceforth assume that $\text{Re}(z)=1$ and investigate first the instance when $\lvert\text{Im}(z)\rvert\leq t^{1/2}\log t.$ A routine application of the formula (\ref{gamma}) with the choice $N=\lceil k/4+3/4\rceil$ then yields
\begin{align*}
\log \Gamma\Big(\frac{1}{4}-\alpha_{j}+\beta\Big)
&-\log \Gamma\Big(\frac{1}{4}+\alpha_{j}\Big)=\big(-\frac{1}{4}-\alpha_{j}+\beta\big)\log\Big(\frac{1}{4}-\alpha_{j}+\beta\Big)
\\
&-\Big(-\frac{1}{4}+\alpha_{j}\Big)\log \Big(\frac{1}{4}+\alpha_{j}\Big)+2\alpha_{j}-\beta
\\
&+\sum_{r=1}^{N}c_{r}\Big(\Big(\frac{1}{4}-\alpha_{j}+\beta\Big)^{1-2r}-\Big(\frac{1}{4}+\alpha_{j}\Big)^{1-2r}\Big)+O(t^{-1-2N}).
\end{align*} 
Observe that then following a similar argument we obtain that the above formula equals
\begin{equation}\label{fbeta}h(\alpha_{j})+z\log (-\alpha_{j})/2+i\frac{z^{2}}{4a_{j}t}+\sum_{(u,v)}c'_{2}(u,v)\beta^{u}\alpha_{j}^{-v}+O(t^{-5/2-k/2}),\end{equation} where as above $(u,v)$ runs over the range earlier described for the discussion pertaining to $G_{1}(x,t)$, the coefficients $c_{2}'(u,v)$ are some fixed complex numbers and the function $h(\alpha)$ is  \begin{equation}\label{hhh}h(\alpha)=-\Big(\frac{1}{4}+\alpha\Big)\log \Big(\frac{1}{4}-\alpha\Big)-\Big(\alpha-\frac{1}{4}\Big)\log\Big(\frac{1}{4}+\alpha\Big)+2\alpha.\end{equation} It seems pertinent to observe first that by definition one has 
\begin{equation*}\log (-\alpha_{j})=\log(\lvert a_{j}\rvert t/2)-i\text{sgn}(a_{j})\pi/2.\end{equation*}Consequently, it transpires that by recalling (\ref{Gmi}) and (\ref{AAA}), summing the formula (\ref{fbeta}) over $j$, taking exponentials and multiplying both sides of the equation at hand by $H(-z,t)z^{-1}$ then one obtains the approximation
\begin{align}\label{HH2}(\pi^{k/2}x)^{-z}G_{2}(z,t)H(-z,t)z^{-1}=&
e^{\phi(t)}A(x,t)^{z}F_{2}(z,t)z^{-1}\big(1+\sum_{(u,v)}c_{2}(u,v)z^{u}t^{-v}\big)\nonumber
\\
&+O\Big(x^{-1}t^{-5/2}\lvert z\rvert^{-1}e^{\lvert I_{\bfa}y\rvert/2t -y^{2}/t}\Big),
\end{align}wherein we wrote $$\phi(t)=\sum_{j=1}^{k}h(\alpha_{j})$$ for the sake of concision, and where we remind the reader of the notation $y=\text{Im}(z)$ and the definition (\ref{Htz}). By integrating the above equation over $[1-it^{1/2}\log t, 1+it^{1/2}\log t],$ it transpires that the contribution $C_{2}(x,t)$ stemming from the error term will satisfy $C_{2}(x,t)\ll x^{-1}t^{-2}$.

Before making further progress it seems desirable to shew first the estimate \begin{equation}\label{phiphi}\lvert\psi(t)-e^{\phi(t)}\rvert\ll t^{-1}.\end{equation} To this end it may as well be worth noting that a customary application of the Taylor expansion of $\log (1+w)$ in (\ref{hhh}) then yields
$$h(\alpha)=-\Big(\frac{1}{4}+\alpha\Big)\log(-\alpha)-\Big(\alpha-\frac{1}{4}\Big)\log\alpha+2\alpha+\sum_{v\geq 1}k_{v}\alpha^{-v}$$ for real coefficients $k_{v}$. Therefore, on substituting $\alpha$ by $\alpha_{j}$ in the above formula we get
$$h(\alpha_{j})=i\text{sgn}(a_{j})\pi/4-ia_{j}t\big(\log(\lvert a_{j}\rvert t/2)-1\big)+O(t^{-1}),$$ whence summing over $j$ the above equation and taking exponentials delivers (\ref{phiphi}). One then may replace $e^{\phi(t)}$ by $\psi(t)$ at the cost of summing an error term $O(x^{-1}t^{-1})$.

It might be worth shifting our focus to the case $\lvert y\rvert>t^{1/2}\log t$. We employ as is customary Stirling's formula (\ref{gamma}) and subsequently take real parts to obtain 
\begin{align*}
\log \lvert\Gamma(1/4-\alpha_{j}+\beta)\rvert-\log \lvert\Gamma(1/4+\alpha_{j})\rvert=
&-\frac{\pi}{4}\lvert y-a_{j}t\rvert+\frac{\pi}{4}\lvert a_{j}\rvert t+O\big(\log(t+\lvert y\rvert)\big).
\end{align*}
One then establishes in an analogous manner by changing signs when required in the discussion concerning (\ref{inecua}) the inequality $$-\sum_{j=1}^{k}\lvert y-a_{j}t\rvert+t\sum_{j=1}^{k}\lvert a_{j}\rvert-\xi_{\bfa}y\leq 0$$ to the end of deriving the bound
$$\lvert G_{2}(z,t)\rvert\ll (yt)^{C}e^{\pi\xi_{\bfa}y/4},$$wherein the above line $C>0$ denotes a fixed constant. Therefore, the previous estimate in conjunction with the definition (\ref{Htz}) delivers
$$(\pi^{k/2}x)^{-z}G_{2}(z,t)H(-z,t)z^{-1}\ll x^{-1}(yt)^{C}e^{-y^{2}/t}.$$ Likewise, an analogous argument reveals that
$$A(x,t)^{z}e^{z^{2}/t+iI_{\bfa}z^{2}/(4t)}z^{-1}\big(1+\sum_{(u,v)}c_{2}(u,v)z^{u}t^{-v}\big)\ll x^{-1}(yt)^{C}e^{\lvert yI_{\bfa}\rvert/(2t)-y^{2}/t}.$$ We integrate (\ref{HH2}) over the lines $\text{Re}(z)=1$ with $\lvert y\rvert> t^{1/2}\log t$ and utilise the bound cognate to $C_{2}(x,t)$ in conjunction with the above inequalities to obtain (\ref{ecI}) for the case $m=2$, as desired.
\end{proof}

As a prelude to the examination of the diagonal and off-diagonal solutions it seems desirable to prepare the ground by discussing certain approximations for the objects introduced in the previous analysis. To this end we recall the reader of (\ref{AAA}), denote henceforth $\psi_{1}(t)=1$ and $\psi_{2}(t)=\psi(t)$ and write
\begin{equation}\label{kj}W_{m}(x,t)=\psi_{m}(t)\Big(K_{m}(x,t)+\sum_{(u,v)}c_{m}(u,v)K_{m,u,v}(x,t)\Big),\ \ \ \ m=1,2,\end{equation} wherein
\begin{equation}\label{Kmuv}K_{m,u,v}(x,t)=\frac{1}{2\pi i}\int_{1-i\infty}^{1+i\infty}A(x,t)^{z}F_{m}(z,t)z^{u-1}t^{-v}dz,\end{equation} and $$K_{m}(x,t)=K_{m,0,0}(x,t),$$ where on the above line the range of summation taken was earlier described right after (\ref{ga}). It may also seem worth introducing beforehand the parameters
\begin{equation}\label{labelabe}\alpha_{\bfa}=\frac{1}{8+I_{\bfa}^{2}},\ \ \ \ \ \ \ \ C_{\bfa}=\alpha_{\bfa}\big(1-\alpha_{\bfa}(1+I_{\bfa}^{2}/8)\big)=\frac{7\alpha_{\bfa}}{8},\end{equation} wherein the reader might find it useful to recall (\ref{ISP}).
\begin{lem}\label{lem0.2}
Let $(u,v)$ lie on the range right after (\ref{ga}). Then it follows that
\begin{equation}\label{grie}K_{m,u,v}(x,t)\ll t^{u/2-v}e^{-C_{\bfa}t(\log A(x,t))^{2}/2}, \ \ \ \  \ \ \ \ m=1,2\end{equation} and 
\begin{equation}\label{JmJmJm}K_{m}(x,t)\ll \log t\end{equation} whenever $\vert\log A(x,t)\rvert\ll t^{-1/2}\log t$. Likewise, one has \begin{equation}\label{raza}K_{m}(x,t)=H(x,t)+O\big(e^{-C_{\bfa}t(\log A(x,t))^{2}/2}\big)\end{equation} if $\lvert \log A(x,t)\rvert\gg t^{-1/2}$, wherein the function $H(x,t)$ is defined as $H(x,t)=1$ if $A(x,t)>1$ and $H(x,t)=0$ if $A(x,t)<1$.
\end{lem}
\begin{proof}
For the sake of concision we omit henceforth the dependence on $x$ and $t$ in $A(x,t)$ and just write $A$. We denote first for convenience $y=\text{Im}(z)$, recall the definition (\ref{F1F1}) and observe that when $\text{Re}(z)=-\alpha_{\bfa}t(\log A)$ one has
$$A^{z}F_{m}(z,t)z^{u-1}t^{-v}\ll t^{-v}\big(\lvert t\log A\rvert^{u-1}+y^{u-1}\big)e^{f(y)-y^{2}/2t},$$ 
where $$f(y)=(-\alpha_{\bfa}+\alpha_{\bfa}^{2})t(\log A)^{2}+\alpha_{\bfa}\lvert y(\log A)\rvert\lvert I_{\bfa}\rvert/2-y^{2}/2t.$$
We find it pertinent to observe that the maximum value of the function $f(y)$ is $-C_{\bfa}t(\log A)^{2}$, the constant $C_{\bfa}>0$ defined above being positive, whence
\begin{equation}\label{tlogtlog}A^{z}F_{m}(z,t)z^{u-1}t^{-v}\ll t^{-v}\big(\lvert t\log A\rvert^{u-1}+y^{u-1}\big)e^{-C_{\bfa}t(\log A)^{2}-y^{2}/2t}.\end{equation} It seems desirable to note first that by differentiating, if needed, one has \begin{equation*}\label{tloga}y^{d}e^{-y^{2}/4t}\ll t^{d/2}\ \ \ \text{and}\ \ \ \ \lvert t\log A\rvert^{d} e^{-C_{\bfa}t(\log A)^{2}/2}\ll t^{d/2}\end{equation*} for every $d>0$. Consequently, integrating on both sides in (\ref{tlogtlog}) and making use of the above bounds for the choice $d=u-1$ it follows that 
\begin{align*}\int_{-\alpha_{\bfa}t(\log A)-i\infty}^{-\alpha_{\bfa}t(\log A)+i\infty}A^{z}F_{m}(z,t)z^{u-1}t^{-v}dz\ll&
 t^{u/2-v-1/2}e^{-C_{\bfa}t(\log A)^{2}/2}\int_{-\infty}^{\infty}e^{-y^{2}/4t}dy,
\end{align*} 
whence a change of variables in the above integrals enables one to conclude that the integral on the left side is $O( t^{u/2-v}e^{-C_{\bfa}t(\log A)^{2}/2})$. It is convenient to observe as well that the integrand in the definition of $K_{m,u,v}(x,t)$ is an entire function, whence we can move the line of integration to $\text{Re}(z)=-\alpha_{\bfa}t(\log A)$ and use the above estimate to obtain (\ref{grie}). The analysis pertaining to $K_{m}(x,t)$ shall be similar to the previous one. We observe first that when $\text{Re}(z)=1$ and $\lvert \log A\rvert\ll t^{-1/2}(\log t)$ then 
\begin{equation*}A^{z}F_{m}(z,t)z^{-1}\ll Ae^{\lvert I_{\bfa}y\rvert/(2t)-y^{2}/2t}\big(1+\lvert y\rvert\big)^{-1}\ll e^{\lvert I_{\bfa}y\rvert/(2t)-y^{2}/2t}\big(1+\lvert y\rvert\big)^{-1},\end{equation*}
whence utilising the above bound and integrating accordingly we deduce the desired estimate (\ref{JmJmJm}). It also transpires that whenever the bound $\lvert \log A\rvert\gg t^{-1/2}$ holds then one has on the line $\text{Re}(z)=-\alpha_{\bfa}(\log A)t$ the estimate
\begin{equation*}A^{z}F_{m}(z,t)z^{-1}\ll e^{-C_{\bfa}t(\log A)^{2}-y^{2}/2t}\big(\lvert t\log A\rvert+\lvert y\rvert\big)^{-1}.\end{equation*}
Therefore, integrating on both sides of the above estimate over the line at hand yields
\begin{align*}\int_{-\alpha_{\bfa}t(\log A)-i\infty}^{-\alpha_{\bfa}t(\log A)+i\infty}A^{z}F_{m}(z,t)z^{-1}dz&\ll  e^{-C_{\bfa}t(\log A)^{2}/2}.\end{align*}

It is worth noting that whenever $A>1$ then $-\alpha_{\bfa}t(\log A)<0$, whence under such circumstances the function on the left side of the above equation has a single pole at $z=0$ in the region between the lines $\text{Re}(z)=1$ and $\text{Re}(z)=-\alpha_{\bfa}t(\log A)$ of residue $1$. If, on the contrary $A<1$ then the aforementioned function does not possess a pole in such a region. Consequently, the preceding discussion in conjunction with the above estimate yields for $\lvert \log A\rvert\gg t^{-1/2}$ the expression (\ref{raza}), which in turn completes the proof.
\end{proof}

\section{Diagonal contribution}\label{sec6.4}

In the present section we integrate the approximate functional equation obtained in Lemma \ref{lem0ita}. We shall henceforth take $\mathbf{a}=(a,-b,-c)$, wherein $a,b,c\in\mathbb{R}$ have the property $a<c\leq b$, such tuples $\mathbf{a}$ satisfying the required inequalities (\ref{AA}). We then utilise the aforementioned lemma and integrate over $[0,T]$ to obtain
\begin{equation}\label{ecuacionp}I_{a,b,c}(T)=I_{1}(T)+I_{2}(T)+O(\log T),\end{equation} where $$I_{m}(T)=\sum_{\bfn\in\mathbb{N}^{3}}P_{\bfn}^{-1/2}\int_{0}^{T}L_{\bfa}(\bfn)^{(-1)^{m+1}it}W_{m}\big(P_{\bfn},t\big)dt, \ \ \ \ \ \ \ \ \ m=1,2.$$ 
We make a distinction between the diagonal and the off-diagonal contribution and write \begin{equation}\label{I111I245}I_{1}(T)=I_{1,1}(T)+I_{1,2}(T),\end{equation} where
$$I_{1,1}(T)=\sum_{n_{1}^{a}=n_{2}^{b}n_{3}^{c}}(n_{1}n_{2}n_{3})^{-1/2}\int_{0}^{T}L_{\bfa}(\bfn)^{it}W_{1}\big(P_{\bfn},t\big)dt$$ and
\begin{equation}\label{I12I12}I_{1,2}(T)=\sum_{n_{1}^{a}\neq n_{2}^{b}n_{3}^{c}}(n_{1}n_{2}n_{3})^{-1/2}\int_{0}^{T}L_{\bfa}(\bfn)^{it}W_{1}\big(P_{\bfn},t\big)dt.\end{equation}
The following lemma will convey the asymptotic evaluation of $I_{1,1}(T)$.
\begin{lem}\label{lemmaiin}
With the above notation one has $$I_{1,1}(T)=\sigma_{a,b,c}T+O\big(T^{1/4+3a/2(a+c)}\log T+T^{1/2}(\log T)^{2}\big).$$
\end{lem}
\begin{proof}
On recalling (\ref{kj}) and the bound $K_{m,u,v}(P_{\bfn},t)\ll t^{-1/2}$ latent in the conclusions of Lemma \ref{lem0.2}, where $(u,v)$ lies in the range described right after (\ref{ga}), we first note that the contribution to the integral arising from the terms $K_{1,u,v}(P_{\bfn},t)$ in the decomposition cognate to $W_{1}(P_{\bfn},t)$ is bounded above by
\begin{align}\label{k1u}\sum_{n_{1}^{a}=n_{2}^{b}n_{3}^{c}}(n_{1}n_{2}n_{3})^{-1/2}\int_{0}^{T}t^{-1/2}dt&
\ll \sum_{(n_{2},n_{3})\in\mathbb{N}^{2}}n_{2}^{-1/2-b/2a}n_{3}^{-1/2-c/2a}\int_{0}^{T}t^{-1/2}dt\ll T^{1/2}.\end{align}
It may also seem pertinent to define for further convenience in the memoir the parameters \begin{equation}\label{tautau}\tau_{\bfn}=2\pi (n_{1}n_{2}n_{3})^{2/3}Q_{\bfa}^{-1/3}.\end{equation} It is then worth observing that employing this notation and making the choice $x=n_{1}n_{2}n_{3}$ in (\ref{AAA}) one has $A(x,t)=(t/\tau_{\bfn})^{3/2}$, whence whenever $\lvert t-\tau_{\bfn}\rvert< t^{1/2}\log t$ it appears at first glance that \begin{equation}\label{fundaA}\frac{2}{3}\log A=(t-\tau_{\bfn})/t+O\big(t^{-1}(\log t)^{2}\big).\end{equation} Then combining (\ref{kj}) and (\ref{k1u}) one gets
\begin{align*}I_{1,1}(T)=&\sum_{n_{1}^{a}=n_{2}^{b}n_{3}^{c}}P_{\bfn}^{-1/2}\int_{\lvert t-\tau_{\bfn}\rvert\geq  t^{1/2}(\log t)}K_{1}(P_{\bfn},t)dt
\\
&+\sum_{n_{1}^{a}=n_{2}^{b}n_{3}^{c}}P_{\bfn}^{-1/2}\int_{\lvert t-\tau_{\bfn}\rvert<  t^{1/2}(\log t)}K_{1}(P_{\bfn},t)dt+O(T^{1/2}),
\end{align*}
wherein we omitted writing the endpoints $0$ and $T$ in the above integrals for the sake of concision. The reader may note that an application of Lemma \ref{lem0.2} in conjunction with the procedure employed to derive (\ref{k1u}) enables one to infer that the second summand on the above equation is $O\big(T^{1/2}(\log T)^{2}\big).$ Likewise, it may be worth observing that whenever $\lvert t-\tau_{\bfn}\rvert\geq  t^{1/2}(\log t)$ then $\lvert \log A\rvert\gg t^{-1/2}(\log t)$, a subsequent application of Lemma \ref{lem0.2} thus delivering
$$I_{1,1}(T)=\sum_{n_{1}^{a}=n_{2}^{b}n_{3}^{c}}P_{\bfn}^{-1/2}\int_{\lvert t-\tau_{\bfn}\rvert\geq  t^{1/2}(\log t)}H(P_{\bfn},t)dt+O\big(T^{1/2}(\log T)^{2}\big),$$whence the definitions of $H(x,t)$ and $\tau_{\bfn}$ then yield
\begin{equation*}\label{ast}I_{1,1}(T)=\sum_{\substack{n_{1}^{a}=n_{2}^{b}n_{3}^{c}\\ \tau_{\bfn}\leq T}}P_{\bfn}^{-1/2}(T-\tau_{\bfn})+O\big(T^{1/2}(\log T)^{2}\big).\end{equation*}

We rewrite the above equation as 
\begin{equation*}\label{ecuacion1}I_{1,1}(T)=\sigma_{a,b,c}T-A_{1}(T)-A_{3}(T)+O\big(T^{1/2}(\log T)^{2}\big),\end{equation*} where we denote
\begin{equation*}\label{0.12ega}A_{1}(T)=\sum_{\substack{\tau_{\bfn}\leq T\\ n_{1}^{a}=n_{2}^{b}n_{3}^{c}}}\tau_{\bfn}P_{\bfn}^{-1/2},\ \ \ \ \ \ \ \ \ \ \ A_{3}(T)=T\sum_{\substack{\tau_{\bfn}>T\\ n_{1}^{a}=n_{2}^{b}n_{3}^{c}}}P_{\bfn}^{-1/2},\end{equation*}
and where the constant $\sigma_{a,b,c}$ is defined by means of \begin{equation}\label{sigma}\sigma_{a,b,c}=\sum_{n_{1}^{a}=n_{2}^{b}n_{3}^{c}}P_{\bfn}^{-1/2},\end{equation} the convergence of the preceding series having its reliance on the fact that $a<\min(b,c)$. Before progressing in the proof we find it pertinent to note that one may derive sharper estimates for $A_{1}(T),A_{3}(T)$ that would have refined the error term in the present lemma if one parametrizes the equation at hand when $a,b,c\in\mathbb{N}$, such an improvement not having any impact in the main theorem of the paper. By a straight substitution we note that
\begin{align}\label{jodo}A_{1}(T)\ll \sum_{n_{2}^{a+b}n_{3}^{a+c}\ll T^{3a/2}}n_{2}^{(a+b)/6a}n_{3}^{(a+c)/6a}&\ll T^{1/4+3a/2(a+c)}\sum_{n_{2}^{a+b}\ll T^{3a/2}}n_{2}^{-(a+b)/(a+c)}\nonumber
\\
&\ll T^{1/4+3a/2(a+c)}\log T.\end{align} Likewise, it transpires that
\begin{align}\label{jodi}A_{3}(T)\ll T\sum_{n_{2}^{a+b}n_{3}^{a+c}\gg T^{3a/2}}n_{2}^{-1/2-b/2a}n_{3}^{-1/2-c/2a}\ll T^{1/4+3a/2(a+c)}\log T,
\end{align}
as desired.
\end{proof}

\section{Off-diagonal contribution}\label{sec6.5}
We now focus our attention on the term $I_{1,2}(T).$ We find it appropiate to consider \begin{equation}\label{Im1}J_{1,u,v}(T)=\sum_{\substack{\bfn\in\mathbb{N}^{3}\\ n_{1}^{a}\neq n_{2}^{b}n_{3}^{c}}}P_{\bfn}^{-1/2}\int_{0}^{T}L_{\bfa}(\bfn)^{it}K_{1,u,v}\big(P_{\bfn},t\big)dt,\ \ \ \ \ \ \ \end{equation} where we remind the reader that $K_{m,u,v}(P_{\bfn},t)$ was defined right after (\ref{kj}). It may be worth introducing the analogous sum
\begin{equation}\label{Justi}J_{1,2}(T)=\sum_{\substack{\bfn\in\mathbb{N}^{3}\\ n_{1}^{a}\neq n_{2}^{b}n_{3}^{c}}}P_{\bfn}^{-1/2}\int_{0}^{T}L_{\bfa}(\bfn)^{it}K_{1}(P_{\bfn},t)dt\end{equation} and observe that equipped with this notation and (\ref{kj}) we may write
\begin{equation*}\label{I12I13}I_{1,2}(T)=J_{1,2}(T)+\sum_{(u,v)}c_{1}(u,v)J_{1,u,v}(T),\end{equation*} where $(u,v)$ runs over the range described right after (\ref{ga}).
\begin{lem}\label{lem715}
For $(u,v)$ in the range described right after (\ref{ga}) one has
$$J_{1,u,v}(T)\ll T^{3/4}(\log T)^{3}.$$
\end{lem}

\begin{proof}
We observe that an application of Lemma \ref{lem0.2} in conjunction with (\ref{fundaA}) then yields $$J_{1,u,v}(T)\ll \sum_{\substack{\tau_{\bfn}\ll T}}P_{\bfn}^{-1/2}\int_{0}^{T}t^{-1/2}e^{-C_{\bfa}t(\log A)^{2}/2}dt+O(T^{-2}),$$ wherein we utilised the fact that the exponential decay stemming from the aforementioned lemma permits one to confine our analysis to the instances $\tau_{\bfn}\ll T.$ By decomposing the above integral into pieces we derive
$$J_{1,u,v}(T)\ll \sum_{\substack{\tau_{\bfn}\ll T}}P_{\bfn}^{-1/2}\int_{\lvert\log A\rvert\leq t^{-1/2}(\log t)}t^{-1/2}e^{-C_{\bfa}t(\log A)^{2}/2}dt+\sum_{\substack{\tau_{\bfn}\ll T}}P_{\bfn}^{-1/2},$$ wherein we omitted writing the endpoints $0$ and $T$. It therefore transpires by alluding to (\ref{fundaA}) that then 
$$J_{1,u,v}(T)\ll \sum_{\substack{\tau_{\bfn}\ll T}}P_{\bfn}^{-1/2}\int_{\lvert t-\tau_{\bfn}\rvert\ll t^{1/2}\log t}t^{-1/2}dt+\sum_{\substack{\tau_{\bfn}\ll T}}P_{\bfn}^{-1/2}\ll \log T\sum_{\substack{\tau_{\bfn}\ll T}}P_{\bfn}^{-1/2}.$$
The reader may also observe that
\begin{equation}\label{Pn}\sum_{P_{\bfn}\ll T^{3/2}}P_{\bfn}^{-1/2}\ll T^{3/4}(\log T)^{2},\end{equation}
whence combining the preceding estimates yields the desired result.
\end{proof}
In order to make progress in the proof it seems pertinent to shift our focus to the corresponding analysis of $J_{1,2}(T)$, both a dyadic argument and a distinction between the contributions stemming from the set of $t$ that are close to $\tau_{\bfn}$ and the one comprising $t$ which are far apart being required. For such purposes we consider for $Q\leq T$ the sets
\begin{equation}\label{SNN}\mathcal{S}_{\bfn}=\Big\{t\in [Q/2,Q]: \ \lvert t-\tau_{\bfn}\rvert\leq Q^{1/2}\log Q\Big\},\end{equation}  \begin{equation}\label{SNN1}\tilde{\mathcal{S}}_{\bfn}=\Big\{t\in [Q/2,Q]: \ \lvert t-\tau_{\bfn}\rvert>Q^{1/2}\log Q\Big\}.\end{equation}
We also find it worth writing \begin{equation*}\label{J1J233}\sum_{\substack{ n_{1}^{a}\neq n_{2}^{b}n_{3}^{c}}}P_{\bfn}^{-1/2}\int_{Q/2}^{Q}L_{\bfa}(\bfn)^{it}K_{1}(P_{\bfn},t)dt=B(Q)+\tilde{B}(Q),\end{equation*} where in the preceding line the summands involved are defined by means of the formulas
\begin{equation}\label{prius}B(Q)=\sum_{\substack{ n_{1}^{a}\neq n_{2}^{b}n_{3}^{c}}}P_{\bfn}^{-1/2}I_{\mathcal{S}_{\bfn}}(Q),\ \ \ \ \ \ \ \ \ \ \tilde{B}(Q)=\sum_{\substack{n_{1}^{a}\neq n_{2}^{b}n_{3}^{c}}}P_{\bfn}^{-1/2}I_{\tilde{\mathcal{S}}_{\bfn}}(Q)\end{equation} with the term $I_{\mathcal{S}}(Q)$ being \begin{equation}\label{prius22}I_{\mathcal{S}}(Q)=\int_{\mathcal{S}}L_{\bfa}(\bfn)^{it}K_{1}(P_{\bfn},t)dt,\ \ \ \ \ \ \ \ \ \ \ \ \ \ \ \ \mathcal{S}=\mathcal{S}_{\bfn},\tilde{\mathcal{S}}_{\bfn}.\end{equation} It seems worth recording for future use and upon recalling (\ref{Justi}) that then
\begin{equation}\label{just}J_{1,2}(T)=\sum_{j=0}^{\big\lfloor \frac{\log T}{\log 2}\big\rfloor}\big(B(2^{-j}T)+\tilde{B}(2^{-j}T)\big)+O\big(T^{3/4}(\log T)^{2}\big),\end{equation} wherein we employed (\ref{Pn}). We shall focus our attention first on the term $B(Q)$. As was discussed above, we find it worth warning the reader that an application of the trivial bound $I_{\mathcal{S}_{\bfn}}(Q)=O\big( Q^{1/2}(\log Q)^{2}\big)$ shall not be of the sufficient strength required.
\begin{lem}\label{lemita6155}
With the above notation, one has for $Q\leq T$ the bound
$$B(Q)\ll Q^{3/4}(\log Q)^{4}.$$
\end{lem}
\begin{proof}
We begin by observing in view of (\ref{SNN}) that it suffices to consider the contribution to $B(Q)$ of tuples with the property that $\mathcal{S}_{\bfn}\neq\o$, such a condition further entailing
\begin{equation}\label{pau}\tau_{\bfn}\asymp Q.\end{equation} We continue then by furnishing ourselves with some notation. We consider for $\bfn_{2}=(n_{2},n_{3})$ the parameter \begin{equation}\label{Nb}N_{1}=[n_{2}^{b/a}n_{3}^{c/a}].\end{equation} We find it worth writing for each triple $\bfn=(n_{1},\bfn_{2})$ the first entry by means of $n_{1}=N_{1}+r$ for some $r\in\mathbb{Z}.$ We shall henceforth write $\mathcal{S}_{\bfn_{2},r}$ and $\tau_{\bfn_{2},r}$ to denote $\mathcal{S}_{\bfn}$ and $\tau_{\bfn}$ respectively. It may also seem pertinent to introduce the functions
\begin{equation}\label{decayc}G_{1}(t,y)=e^{(\log A)+1/t+I_{\bfa}y/2t-y^{2}/t}(1+iy)^{-1}\end{equation} and $$F_{\bfn_{2},r}(t,y)=y(\log A)+2y/t-I_{\bfa}/4t+I_{\bfa}y^{2}/4t+t\log\big(L_{\bfa}(\bfn)\big).$$ 
We note first for further use that whenever $t\in \mathcal{S}_{\bfn}$ then $\lvert \log A\rvert\ll t^{-1/2}(\log t)$, whence
\begin{equation}\label{boundfG}G_{1}(t,y)\ll (1+\lvert y\rvert)^{-1}.\end{equation}

In view of the above equations, we note that for fixed $y$ and $\bfn$ the zeros of the function \begin{equation*}\label{monomono}\frac{d}{dt}\big(G_{1}(t,y)F_{\bfn_{2},r}'(t,y)^{-1}\big)\end{equation*} are also zeros of a function $$P\big(t,y,\log A, \log(L_{\bfa}(\bfn))\big ),$$ wherein $P_{1}(z_{1},z_{2},z_{3},z_{4})$ is a polynomial of degree smaller than $C$ for some universal constant $C>0$. It therefore transpires that when thinking of $y$ and $\bfn$ as being fixed then subsequent applications of Rolle's theorem enables one to partition the set of integration into a bounded number of intervals (not depending on $y$) in which $G_{1}(t,y)F_{\bfn_{2},r}'(t,y)^{-1}$ is monotonic. By recalling (\ref{prius22}) and in view of the decay exhibited by $G_{1}(t,y)$ with respect to $y$ in (\ref{decayc}), one has that
$$I_{\mathcal{S}_{\bfn}}(Q)=\int_{\mathcal{S}_{\bfn}}\int_{-Q^{1/2}\log Q}^{Q^{1/2}\log Q}G_{1}(t,y)e^{iF_{\bfn_{2},r}(t,y)}dydt+O(Q^{-2}).$$

 We may suppose that $\mathcal{S}_{\bfn}\neq \o$, since if not no further work would be required. It might be convenient to observe first that whenever $y$ and $t$ lie in the set of integration at hand then it follows that
\begin{align}\label{Fnn}F_{\bfn_{2},r}'(t,y)&
=\frac{3}{2}y/t+\log \big(n_{2}^{b}n_{3}^{c}/(N_{1}+r)^{a}\big)+O\big(t^{-1}(\log t)^{2}\big)\nonumber
\\
&=\frac{3}{2}y/t+\log \big(n_{2}^{b}n_{3}^{c}/N_{1}^{a}\big)-a\log(1+r/N_{1})+O\big(t^{-1}(\log t)^{2}\big).\end{align}
We further write, for convenience, $$H_{\bfn_{2}}(t,y)=F_{\bfn_{2},r}'(t,y)+a\log(1+r/N_{1}),$$ a careful examination of which revealing that it does not depend on $r$. The reader may find it useful to recall the definition of $\mathcal{S}_{\bfn_{2},r}$ and $\tau_{\bfn_{2},r}$ right after (\ref{Nb}) and observe that for fixed $\bfn_{2}$, given $r_{1},r_{2}\in\mathbb{Z}$ satisfying $\lvert r_{1}\rvert,\lvert r_{2}\rvert\leq N_{1}/2$ and $t_{1}\in\mathcal{S}_{\bfn_{2},r_{1}}$ and $t_{2}\in \mathcal{S}_{\bfn_{2},r_{2}}$ then it transpires that 
\begin{equation*}\lvert H_{\bfn_{2}}(t_{1},y)-H_{\bfn_{2}}(t_{2},y)\rvert\ll Q^{-1/2}\log Q,
\end{equation*} the above implicit constant not depending on $r_{1},r_{2},$ and in turn implies that the cardinality of the set $\mathcal{R}_{1}$ comprising integers $\lvert r\rvert\leq N_{1}/2$ with the property that $\lvert F_{\bfn_{2},r}'(t,y)\rvert\leq N_{1}^{-1}$ for some $t\in\mathcal{S}_{\bfn_{2},r}$ satisfies the bound $$\lvert \mathcal{R}_{1}\rvert\ll N_{1}Q^{-1/2}(\log Q)+1.$$ For these cases and upon recalling (\ref{pau}), an application of the trivial bound $Q^{1/2}(\log Q)^{2}$, it in turn stemming, inter alia, from the bound (\ref{boundfG}), to the integral at hand already suffices to bound the contribution arising from the aforementioned set by 
\begin{align*}\sum_{\substack{n_{2}^{a+b}n_{3}^{a+c}\ll Q^{3a/2}\\ r\in\mathcal{R}_{1}}}n_{2}^{-1/2}n_{3}^{-1/2}N_{1}^{-1/2}I_{\mathcal{S}_{\bfn}}(Q)&\ll (\log Q)^{3}\sum_{n_{2}^{a+b}n_{3}^{a+c}\ll Q^{3a/2}}n_{2}^{-1/2}n_{3}^{-1/2}N_{1}^{1/2}\\
&+Q^{1/2}(\log Q)^{2}\sum_{n_{2}^{a+b}n_{3}^{a+c}\ll Q^{3a/2}}n_{2}^{-1/2}n_{3}^{-1/2}N_{1}^{-1/2}
\\
&\ll Q^{3/4}(\log Q)^{3}\sum_{n_{3}\ll Q^{3a/(2(a+c))}}n_{3}^{-1}\ll Q^{3/4}(\log Q)^{4}.
\end{align*}

Moreover, on denoting $\mathcal{R}_{2}$ to the set of numbers $r$ satisfying $\lvert F_{\bfn_{2},r}' (t,y)\rvert> N_{1}^{-1}$ for each $t\in\mathcal{S}_{\bfn_{2},r}$ one further has
$$\sum_{r\in\mathcal{R}_{2}}\lvert F_{\bfn_{2},r}'(t,y)\rvert^{-1}\ll N_{1}\sum_{\lvert r\rvert\leq  N_{1}/2}\frac{1}{r}\ll N_{1}\log Q$$ for fixed $t$. Therefore, the preceding discussion in conjunction with Titchmarsh \cite[Lemma 4.3]{Tit} and equations (\ref{boundfG}) and the subsequent analysis delivers
\begin{align*}\label{MMM}\sum_{\substack{n_{2}^{a+b}n_{3}^{a+c}\ll Q^{3a/2}\\ r\in\mathcal{R}_{2}}}n_{2}^{-1/2}n_{3}^{-1/2}N_{1}^{-1/2}I_{\mathcal{S}_{\bfn}}(Q)\ll Q^{3/4}(\log Q)^{3}. 
\end{align*}

We also remark that for integers with the property that $\lvert r\rvert> N_{1}/2$ one then further has $\lvert\log(L_{\bfa}(\bfn))\rvert\gg 1$, an immediate consequence of which being when applied in conjunction with the observation that the rest of the summands in (\ref{Fnn}) are $O(Q^{-1/2}\log Q)$ that then $\lvert F_{\bfn_{2},r}'(t,y)\rvert\gg 1.$ Therefore, combining the previous discussion with another application of Titchmarsh \cite[Lemma 4.3]{Tit} and the analysis following (\ref{boundfG}) we derive that such a contribution would then be $O(Q^{3/4}(\log Q)^{3}).$
\end{proof}
We next shift our focus to the analysis of the term $\tilde{B}(Q)$, it being convenient for such purposes recalling (\ref{I12I12}), (\ref{prius}) and (\ref{Nb}) and introducing for pairs $(n_{2},n_{3})\in\mathbb{N}^{2}$ the function
\begin{equation}\label{elite}L(n_{2},n_{3})=\log(n_{2}^{b}n_{3}^{c}/N_{1}^{a}).\end{equation}
\begin{lem}\label{lemita616}
One has that
$$I_{1,2}(T)=\sum_{\substack{\tau_{\bfn}\leq T\\ n_{1}=N_{1}}}\frac{P_{\bfn}^{-1/2}}{L(n_{2},n_{3})}\Big(e\big(TL(n_{2},n_{3})\big)-e\big(\tau_{\bfn}L(n_{2},n_{3})\big)\Big)+O\big(T^{3/4}(\log T)^{3}\big).$$

\end{lem}

\begin{proof}

We start by observing in view of (\ref{SNN1}), (\ref{prius22}) and Lemma \ref{lem0.2} that then
$$I_{\tilde{\mathcal{S}}_{\bfn}}(Q)=\int_{\tilde{\mathcal{S}}_{\bfn}\cap [\tau_{\bfn},Q]}L_{\bfa}(\bfn)^{it}dt+O(Q^{-2}).$$
Then upon recalling (\ref{Nb}) we write $n_{1}=N_{1}+r$ for $r\neq 0$ and note that whenever $1\leq\lvert r\rvert \leq N_{1}/2$ then \begin{equation*}\label{log35}\big\lvert\log(L_{\bfa}(\bfn))\big\rvert^{-1}\asymp \frac{N_{1}^{a}}{\lvert (N_{1}+r)^{a}-n_{2}^{b}n_{3}^{c}\rvert}\asymp \frac{N_{1}}{\lvert r\rvert}.\end{equation*} It may be appropiate to denote  $\tilde{B}_{1}(Q)$ the contribution to $\tilde{B}(Q)$ stemming from tuples satisfying $\lvert n_{1}-N_{1}\rvert\geq 1 ,$ and thus write
\begin{equation*}\label{JSEC}\tilde{B}(Q)=\tilde{B}_{1}(Q)+\tilde{B}_{2}(Q),\end{equation*} wherein $\tilde{B}_{2}(Q)$ denotes the corresponding contribution arising from the instance $n_{1}=N_{1}.$ Summing over $1\leq \lvert r\rvert\leq  N_{1}/2$ and combining the above equations and the procedure in the preceding lemma delivers
\begin{align*}\label{JSN1} \tilde{B}_{1}(Q)&\ll \sum_{\substack{n_{2}^{a+b}n_{3}^{a+c}\ll Q^{3a/2}\\ 1\leq\lvert r\rvert\leq  N_{1}/2}}n_{2}^{-1/2}n_{3}^{-1/2}N_{1}^{-1/2} I_{\tilde{\mathcal{S}}_{\bfn}}(Q)\ll Q^{3/4}(\log Q)^{2}.
\end{align*}
 
Likewise, it may be worth noting that whenever $\lvert r\rvert> N_{1}/2$ then $\lvert\log\big(L_{\bfa}(\bfn)\big)\rvert^{-1}\ll 1$, the contribution stemming from triples satisfying such a property being $O(Q^{3/4}(\log Q)^{2})$ in view of (\ref{Pn}). The preceding discussion then yields the formula $$\tilde{B}(Q)=\sum_{\substack{\tau_{\bfn}\leq Q\\  n_{1}=N_{1}}}P_{\bfn}^{-1/2}\int_{\tilde{\mathcal{S}}_{\bfn}\cap [\tau_{\bfn},Q]}L_{\bfa}(\bfn)^{it}dt+O\big(Q^{3/4}(\log Q)^{2}\big),$$ whence recalling (\ref{just}) and Lemma \ref{lemita6155} and summing over dyadic intervals enables one to derive
$$J_{1,2}(T)=\sum_{\substack{\tau_{\bfn}\leq T\\  n_{1}=N_{1}}}P_{\bfn}^{-1/2}\Bigg(\int_{\tau_{\bfn}}^{T}L_{\bfa}(\bfn)^{it}dt-\int_{\mathcal{C}}L_{\bfa}(\bfn)^{it}dt\Bigg)+O\big(T^{3/4}(\log T)^{3}\big),$$ wherein $\mathcal{C}$ is a set satisfying $\lvert \mathcal{C}\rvert\ll T^{1/2}(\log T).$ We see from the definition (\ref{Nb}) and the fact that $a<\min (b,c)$ that then
$$\sum_{\substack{\tau_{\bfn}\leq T\\  n_{1}=N_{1}}}P_{\bfn}^{-1/2}\ll \sum_{\substack{n_{2}^{a+b}n_{3}^{a+c}\ll T^{3a/2}}}n_{2}^{-1/2-b/2a}n_{3}^{-1/2-c/2a}\ll 1,$$ such an observation when combined with the preceding equation thus delivering
$$J_{1,2}(T)=\sum_{\substack{\tau_{\bfn}\leq T\\  n_{1}=N_{1}}}P_{\bfn}^{-1/2}\int_{\tau_{\bfn}}^{T}L_{\bfa}(\bfn)^{it}dt+O\big(T^{3/4}(\log T)^{3}\big),$$ as desired. The result then follows upon recalling (\ref{elite}) by computing the above integral accordingly and applying Lemma \ref{lem715}.
\end{proof}

Further progress in the course of the argumentation hinges on a reappraisal of some of the terms stemming in the analysis deployed in \cite{Pli}, it being required to such an end to introduce first some notation. We write as in \cite[(3.2)]{Pli} and for every triple $\bfn\in\mathbb{N}^{3}$ the parameters
$$N_{\bfn}=2\pi\max(n_{1}^{2}/a,n_{2}^{2}/b,n_{3}^{2}/c\big),\ \ \ \ \ \ \ T_{1}=T/2\pi,$$ and introduce the sum defined right above \cite[Lemma 6.2]{Pli} by means of 
\begin{equation}\label{M1T}M_{1}(T)=2\pi a^{-1}\mathop{{\sum_{N_{\bfn}\leq \frac{2\pi}{a} n_{1}n_{2}^{b/a}n_{3}^{c/a}\leq T}}^*}n_{2}^{b/2a-1/2}n_{3}^{c/2a-1/2}e(n_{1}n_{2}^{b/a}n_{3}^{c/a}),\end{equation} wherein the above sum $n_{2}^{b/a}n_{3}^{c/a}$ is not an integer. It also seems worth considering
\begin{equation}\label{integr}J_{2,2}(T)=\sum_{\substack{N_{\bfn}\leq T\\ n_{1}=N_{1}}}P_{\bfn}^{-1/2}\int_{N_{\bfn}}^{T}e^{itL(n_{2},n_{3})}dt,\end{equation} and, for pairs $\bfn_{2}=(n_{2},n_{3})\in\mathbb{N}^{2}$ the function
$$g(\bfn_{2})=\max(ac^{-1}n_{3}^{2-c/a}n_{2}^{-b/a},ab^{-1}n_{2}^{2-b/a}n_{3}^{-c/a}).$$
The upcoming technical lemma shall be required to achieve the aforementioned endeavour.
\begin{lem}\label{lem6.4}
Let $Q_{i},P_{i}: \mathbb{N}^{2}\rightarrow \mathbb{R}$ for $i=1,2$ be real valued functions having the property for $\bfn_{2}\in\mathbb{N}^{2}$ that 
$$\lvert\chi_{1}(\bfn_{2})-\chi_{2}(\bfn_{2})\rvert\ll \| n_{2}^{b/a}n_{3}^{c/a}\|,\ \ \ \ \ \ \ \chi_{i}=P_{i},Q_{i},\ \ \ \ i=1,2,$$ and such that $P_{i}(\bfn_{2})\asymp n_{2}^{b/a}n_{3}^{c/a}$ for $i=1,2$. Likewise, let $R_{i}: \mathbb{N}^{2}\rightarrow \mathbb{R}$ for $i=1,2$ be another pair of functions satisfying $R_{i}(\bfn_{2})\asymp \| n_{2}^{b/a}n_{3}^{c/a}\|$ for $i=1,2$ and
$$R_{1}(\bfn_{2})-R_{2}(\bfn_{2})\ll \| n_{2}^{b/a}n_{3}^{c/a}\|^{2}.$$ Moreover, let $$\mathcal{A}\subset\Big\{(n_{2},n_{3})\in\mathbb{N}^{2}:\ \ n_{2}^{b/a}n_{3}^{c/a}\leq aT_{1}\Big\}.$$ Then, upon defining the weighted exponential sum
$$S_{i}(\bfn_{2})=\sum_{\bfn_{2}\in\mathcal{A}}n_{2}^{-1/2}n_{3}^{-1/2}P_{i}(\bfn_{2})^{1/2}\frac{e\big(Q_{i}(\bfn_{2})\big)}{R_{i}(\bfn_{2})}\ \ \ \ \ \ \ \ \ i=1,2,$$
one has that 
$$S_{1}(\bfn_{2})-S_{2}(\bfn_{2})\ll T^{1/2+a/2c}\log T.$$
\end{lem}
\begin{proof}
The proof has its reliance on the application of both the above estimates for the corresponding differences in conjunction with the mean value theorem and the bound \begin{equation}\label{estimaci}\sum_{n_{2}^{b/a}n_{3}^{c/a}\leq aT_{1}}n_{2}^{(b-a)/2a}n_{3}^{(c-a)/2a}\ll T^{(b+a)/2b}\sum_{n_{3}^{c/a}\leq aT_{1}}n_{3}^{-(b+c)/2b}\ll T^{(a+c)/2c}\log T.\end{equation} 
More precisely, \begin{align*}&S_{1}(\bfn_{2})-\sum_{\bfn_{2}\in\mathcal{A}}n_{2}^{-1/2}n_{3}^{-1/2}P_{1}(\bfn_{2})^{1/2}\frac{e\big(Q_{2}(\bfn_{2})\big)}{R_{1}(\bfn_{2})}
\\
&\ll \sum_{n_{2}^{b/a}n_{3}^{c/a}\leq aT_{1}}n_{2}^{-1/2}n_{3}^{-1/2}\frac{P_{1}(\bfn_{2})^{1/2}\| n_{2}^{b/a}n_{3}^{c/a}\|}{R_{1}(\bfn_{2})}\ll \sum_{n_{2}^{b/a}n_{3}^{c/a}\leq aT_{1}}n_{2}^{b/2a-1/2}n_{3}^{c/2a-1/2},
\end{align*}
the aforementioned use of the mean value theorem being the genesis of the first step and the application of (\ref{estimaci}) combined with the assumptions on the sizes of the corresponding functions permitting one to deduce that the above sum is $O(T^{1/2+a/2c}\log T)$. Similarly,
\begin{align*}\sum_{\bfn_{2}\in\mathcal{A}}&n_{2}^{-1/2}n_{3}^{-1/2}P_{1}(\bfn_{2})^{1/2}e\big(Q_{2}(\bfn_{2})\big)\Big(\frac{1}{R_{1}(\bfn_{2})}-\frac{1}{R_{2}(\bfn_{2})}\Big)
\\
&\ll \sum_{n_{2}^{b/a}n_{3}^{c/a}\leq aT_{1}}n_{2}^{b/2a-1/2}n_{3}^{c/2a-1/2}\ll T^{1/2+a/2c}\log T,
\end{align*} wherein we employed (\ref{estimaci}). The same principle permits one to derive the estimate
\begin{align*}\sum_{\bfn_{2}\in\mathcal{A}}&n_{2}^{-1/2}n_{3}^{-1/2}\frac{e\big(Q_{2}(\bfn_{2})\big)}{R_{2}(\bfn_{2})}\Big(P_{1}(\bfn_{2})^{1/2}-P_{2}(\bfn_{2})^{1/2}\Big)\ll 1,
\end{align*}
whence a combination of the preceding bounds enables one to deduce the desired conclusion.
\end{proof}
Equipped with the preceding result we have reached a position from which to prove the following proposition, it being pertinent to recall first (\ref{Nb}) and introduce the functions $$ G(\bfn_{2})=\lceil g(\bfn_{2})\rceil \| n_{2}^{b/a}n_{3}^{c/a}\|,\ \ \ \ \ \ \ \ \  \ \ H(\bfn_{2})=\Bigg[ \frac{aT_{1}}{n_{2}^{b/a}n_{3}^{c/a}}\Bigg],$$ and the set \begin{equation}\label{monic}\mathcal{Z}_{1}=\Big\{(n_{2},n_{3})\in\mathbb{N}^{2}:\ \ \ n_{2}\leq \sqrt{bT_{1}},\ \ \ n_{3}\leq \sqrt{cT_{1}},\ \ \ n_{2}^{b/a}n_{3}^{c/a}\leq aT_{1} \Big\}.\end{equation}
\begin{prop}\label{propo3}
With the above notation one has that
\begin{align*}M_{1}(T)+J_{2,2}(T)=&\frac{1}{i}\sum_{\substack{\bfn_{2}\in\mathcal{Z}_{1}\\ n_{1}=N_{1}}}\frac{P_{\bfn}^{-1/2}}{L(\bfn_{2})}\Big(e\big(T_{1}L(\bfn_{2})\big)-e\big(G(\bfn_{2})\big)\Big)+O(T^{1/2+a/2c}\log T).
\end{align*}
\end{prop}

\begin{proof}
We shall start our endeavour by examining first $M_{1}(T)$ and noting that then the underlying restrictions on the variables can be rephrased as
$$g(\bfn_{2})\leq n_{1}\leq \min\big(aT_{1}/( n_{2}^{b/a}n_{3}^{c/a}),n_{2}^{b/a}n_{3}^{c/a}\big).$$ Consequently, by observing when summing over $n_{1}$ that one is dealing with the terms of a geometric progression it then transpires that
\begin{align}\label{ec9}M_{1}(T)=&\frac{2\pi}{a}\sum_{\substack{[n_{2}^{b/a}n_{3}^{c/a}]\leq \sqrt{aT_{1}}}}n_{2}^{b/2a-1}n_{3}^{c/2a-1/2}\frac{e\big(N_{1} n_{2}^{b/a}n_{3}^{c/a}\big)-e\big(G(\bfn_{2})\big)}{e\big(n_{2}^{b/a}n_{3}^{c/a}\big)-1}+O(T^{1/2+a/2c}\log T)
\\
&+\frac{2\pi}{a}\sum_{\substack{[n_{2}^{b/a}n_{3}^{c/a}]> \sqrt{aT_{1}}}}n_{2}^{b/2a-1/2}n_{3}^{c/2a-1/2}\frac{e\big(H(\bfn_{2}) n_{2}^{b/a}n_{3}^{c/a}\big)-e\big(G(\bfn_{2})\big)}{e\big(n_{2}^{b/a}n_{3}^{c/a}\big)-1}\nonumber,
\end{align}
wherein we omitted as we shall do henceforth writing $(n_{2},n_{3})\in\mathcal{Z}_{1}$, the error term therein stemming from an application of the mean value theorem in conjunction with the bound (\ref{estimaci}) when choosing the endpoint of the interval of summation cognate to $n_{1}$. Likewise, computing the integral accordingly in (\ref{integr}) delivers 
\begin{align}\label{ecc9}J_{2,2}(T)=&-i\sum_{\substack{[n_{2}^{b/a}n_{3}^{c/a}]\leq \sqrt{aT_{1}}\\  n_{1}=N_{1}}}\frac{P_{\bfn}^{-1/2}}{L(\bfn_{2})}\Big(e\big(T_{1}L(\bfn_{2})\big)-e\big(N_{1}^{2}L(\bfn_{2})\big)\Big)+O(1),
\end{align} wherein the preceding sum it is apparent that $N_{1}=\max(N_{1},\bfn_{2})$ when either $n_{2}$ or $n_{3}$ are sufficiently large.

It then seems worth observing when $N_{1}\gg \sqrt{T_{1}}$ that
\begin{align*}aT_{1}\log(n_{2}^{b/a}n_{3}^{c/a}/N_{1})=H(\bfn_{2})\|n_{2}^{b/a}n_{3}^{c/a}\|+O(\| n_{2}^{b/a}n_{3}^{c/a}\|),
\end{align*}
the main term in the last expression in turn satisfying
\begin{equation}\label{congru}H(\bfn_{2})\|n_{2}^{b/a}n_{3}^{c/a}\|\equiv H(\bfn_{2}) n_{2}^{b/a}n_{3}^{c/a}\mmod{1}.\end{equation}
We further anticipate that it is apparent by using the Taylor expansion that
$$\frac{N_{1}^{-1/2}}{L(\bfn_{2})}=\frac{N_{1}^{1/2}}{a\| n_{2}^{b/a}n_{3}^{c/a}\|\big(1+O\big(\| n_{2}^{b/a}n_{3}^{c/a}\| N_{1}^{-1}\big)\big)}$$ and
$$e\big(n_{2}^{b/a}n_{3}^{c/a}\big)-1=2\pi i\| n_{2}^{b/a}n_{3}^{c/a}\|+O\big(\| n_{2}^{b/a}n_{3}^{c/a}\|^{2}\big).$$ The preceding discussion enables one to apply Lemma \ref{lem6.4} and derive
\begin{align}\label{pili}\frac{2\pi}{a}\sum_{\substack{[n_{2}^{b/a}n_{3}^{c/a}]> \sqrt{aT_{1}}}}&n_{2}^{b/2a-1/2}n_{3}^{c/2a-1/2}\frac{e\big(H(\bfn_{2}) n_{2}^{b/a}n_{3}^{c/a}\big)}{e\big(n_{2}^{b/a}n_{3}^{c/a}\big)-1}\nonumber
\\
&=-i\sum_{\substack{[n_{2}^{b/a}n_{3}^{c/a}]> \sqrt{aT_{1}}\\ n_{1}=N_{1}}}\frac{P_{\bfn}^{-1/2}}{L(\bfn_{2})}e\big(T_{1}L(\bfn_{2})\big)+O(T^{1/2+a/2c}\log T).
\end{align}

Likewise, we observe that $$N_{1}^{2}L(\bfn_{2})=N_{1} \| n_{2}^{b/a}n_{3}^{c/a}\|+O\big(\| n_{2}^{b/a}n_{3}^{c/a}\|^{2}\big),$$ whence such a remark and an analogous congruence to that in (\ref{congru}) in conjunction with previous considerations constitute the conditions required for the application of Lemma \ref{lem6.4}, it then entailing
\begin{align}\label{pili2}\frac{2\pi}{a}\sum_{\substack{[n_{2}^{b/a}n_{3}^{c/a}]\leq \sqrt{aT_{1}}}}&n_{2}^{b/2a-1}n_{3}^{c/2a-1/2}\frac{e\big(N_{1} n_{2}^{b/a}n_{3}^{c/a}\big)}{e\big(n_{2}^{b/a}n_{3}^{c/a}\big)-1}\nonumber
\\
&+i\sum_{\substack{[n_{2}^{b/a}n_{3}^{c/a}]\leq \sqrt{aT_{1}}\\ n_{1}=N_{1}}}\frac{P_{\bfn}^{-1/2}}{L(\bfn_{2})}e\big(N_{1}^{2}L(\bfn_{2})\big)\ll T^{1/2+a/2c}\log T.
\end{align}
The lemma then follows by adding equations (\ref{ec9}) and (\ref{ecc9}) and employing both (\ref{pili}) and (\ref{pili2}).
\end{proof}

We have then reached a point from which to present a fundamental proposition in the memoir, it being required beforehand to recall equations (\ref{Justi}), (\ref{M1T}), (\ref{integr}) and (\ref{monic}) to the reader and introduce the set  $$\mathcal{Z}_{2}=\Big\{(n_{2},n_{3})\in\mathbb{N}^{3}:\ \ \ n_{2}n_{3}[n_{2}^{b/a}n_{3}^{c/a}]\leq (T/2\pi)^{3/2}\sqrt{abc}\Big\}$$ and $\mathcal{S}_{1}=\mathcal{Z}_{1}\setminus \mathcal{Z}_{2}$ and $\mathcal{S}_{2}=\mathcal{Z}_{2}\setminus \mathcal{Z}_{1}.$
\begin{prop}\label{maldini} Whenever $a<c\leq b$ one has
$$M_{1}(T)+J_{2,2}(T)-I_{1,2}(T)\ll T^{1/4+\frac{3a}{2(a+c)}}\log T+T^{5/4-c/4a}+T^{1/2+a/2c}\log T+T^{3/4}(\log T)^{3}.$$

\end{prop}
\begin{proof}
We employ Lemma \ref{lemita616} and Proposition \ref{propo3} for the purpose of obtaining
$$M_{1}(T)+J_{2,2}(T)-I_{1,2}(T)=Z_{1}(T)-Z_{2}(T)+Z_{3}(T)+Z_{4}(T)+O\big(T^{3/4}(\log T)^{3}\big),$$
wherein $$Z_{m}(T)=\frac{1}{ai}\sum_{\substack{\bfn_{2}\in \mathcal{S}_{m}}}\frac{P_{\bfn}^{-1/2}}{L(\bfn_{2})}\Big(e\big(T_{1}L(\bfn_{2})\big)-1\Big),\ \ \ \ \ \ \ \ \ m=1,2,$$
$$Z_{3}(T)=\frac{1}{ai}\sum_{\substack{\bfn_{2}\in\mathcal{Z}_{1}}}\frac{P_{\bfn}^{-1/2}}{L(\bfn_{2})}\big(1-e\big(G(\bfn_{2})\big)\big)$$
and
$$Z_{4}(T)=\sum_{\substack{\bfn_{2}\in\mathcal{Z}_{2}}}\frac{P_{\bfn}^{-1/2}}{L(\bfn_{2})}\Big(e\big(\tau_{\bfn}L(\bfn_{2})\big)-1\Big).$$
The treatment of the above terms shall have its reliance on the application of the mean value theorem. We thus begin such an endeavour by obtaining
$$Z_{1}(T)\ll T\sum_{\substack{n_{2}^{b+a}n_{3}^{c+a}\gg T^{3a/2}}}n_{2}^{-1/2-b/2a}n_{3}^{-1/2-c/2a}\ll T^{1/4+\frac{3a}{2(a+c)}}\log T,$$ wherein we employed (\ref{jodo}). An analogous argument enables one to derive the estimate
\begin{align*}Z_{2}(T)&\ll T\sum_{\substack{n_{2}^{b/a}n_{3}^{c/a}\gg T}}n_{2}^{-1/2-b/2a}n_{3}^{-1/2-c/2a}+T\sum_{\substack{n_{3}\gg \sqrt{T}}}n_{3}^{-1/2-c/2a}+T\sum_{\substack{n_{2}\gg \sqrt{T}}}n_{2}^{-1/2-b/2a}
\\
&\ll T^{1/2+a/2c}\log T+T^{5/4-c/4a}.
\end{align*}

Likewise, the same principle permits one to conclude that
\begin{align*}Z_{3}(T)&\ll \sum_{\substack{n_{3}\leq \sqrt{cT_{1}}}}n_{2}^{-1/2-b/2a}n_{3}^{3/2-c/2a}\ll T^{5/4-c/4a}.
\end{align*}
Similarly, by (\ref{jodi}) we get
\begin{align*}Z_{4}(T)\ll \sum_{\substack{n_{2}^{b+a}n_{3}^{c+a}\ll T^{3a/2}}}n_{2}^{1/6+b/6a}n_{3}^{1/6+c/6a}&\ll T^{1/4+\frac{3a}{2(c+a)}}\log T,
\end{align*}
as desired. 
\end{proof}

\section{Residual terms arising from the twisted integral analysis}\label{sec6.6}
The investigations that will be presented herein analysing $I_{2}(T)$ ultimately deliver bounds from residual terms in the spirit of both Lemmata \ref{lem715} and \ref{lemita6155}. We find it appropiate to recall (\ref{F1F1}), (\ref{Kmuv}), (\ref{labelabe}) and consider, as was done in (\ref{Im1}), the sum \begin{equation*}\label{Im2}J_{2,u,v}(T)=\sum_{\bfn\in\mathbb{N}^{3}}P_{\bfn}^{-1/2}\int_{0}^{T}\psi(t)L_{\bfa}(\bfn)^{-it}K_{2,u,v}\big(P_{\bfn},t\big)dt.\ \ \ \ \ \ \ \end{equation*} It may be worth introducing the analogous sum
\begin{equation}\label{J21J21}J_{2,1}(T)=\sum_{\bfn\in\mathbb{N}^{3}}P_{\bfn}^{-1/2}\int_{0}^{T}\psi(t)L_{\bfa}(\bfn)^{-it}K_{2}(P_{\bfn},t)dt\end{equation} and observe that equipped with this notation we may write $I_{2}(T)$ by making use of (\ref{kj}) in a rather concise manner, say
\begin{equation}\label{osh}I_{2}(T)=J_{2,1}(T)+\sum_{(u,v)}c_{2}(u,v)J_{2,u,v}(T),\end{equation} wherein $(u,v)$ lies in the range described right after (\ref{ga}).
\begin{lem}\label{J2uv}
With the above notation, one has
$$J_{2,u,v}(T)\ll T^{3/4}(\log T)^{3}.$$
\end{lem}
\begin{proof}
We observe as in Lemma \ref{lem715} that an application of Lemma \ref{lem0.2} in conjunction with (\ref{fundaA}) then yields $$J_{2,u,v}(T)\ll \sum_{\substack{\tau_{\bfn}\ll T}}P_{\bfn}^{-1/2}\int_{0}^{T}t^{-1/2}e^{-C_{\bfa}t(\log A)^{2}/2}dt+O(T^{-2}),$$ whence the same argument as therein yields the desired result.

\end{proof}
In order to make progress in the proof, it seems pertinent to shift our focus to the contribution to $I_{2}(T)$ stemming from the term $J_{2,1}(T)$. We find it worth anticipating that a dyadic argument shall be required henceforth. To this end and for $Q\leq T$ we write \begin{equation}\label{J1J23}\sum_{\bfn\in\mathbb{N}^{3}}P_{\bfn}^{-1/2}\int_{Q/2}^{Q}\psi(t)L_{\bfa}(\bfn)^{-it}K_{2}(P_{\bfn},t)dt=B^{\psi}(Q)+\tilde{B}^{\psi}(Q),\end{equation} with the above terms on the right side defined by means of
\begin{equation}\label{priuspsi}B^{\psi}(Q)=\sum_{\bfn\in\mathbb{N}^{3}}P_{\bfn}^{-1/2}I_{\mathcal{S}_{\bfn}}^{\psi}(Q),\ \ \ \ \ \ \ \tilde{B}^{\psi}(Q)=\sum_{\bfn\in\mathbb{N}^{3}}P_{\bfn}^{-1/2}I_{\tilde{\mathcal{S}}_{\bfn}}^{\psi}(Q)\end{equation} wherein upon recalling (\ref{SNN}) and (\ref{SNN1}) then \begin{equation*}\label{prius2}I_{\mathcal{S}}^{\psi}(Q)=\int_{\mathcal{S}}\psi(t)L_{\bfa}(\bfn)^{-it}K_{2}(P_{\bfn},t)dt,\ \ \ \ \ \ \ \ \mathcal{S}=\mathcal{S}_{\bfn},\tilde{\mathcal{S}}_{\bfn}.\end{equation*}
Before providing an explicit bound for the sum $B^{\psi}(Q)$ it seems worth presenting first a technical lemma that shall be used on several occasions in subsequent analysis. To this end, we recall the definition (\ref{ga}) and introduce the function \begin{equation}\label{GGGG}G_{\bfn}(t)=-t\log L_{\bfa}(\bfn)-g_{\bfa}(t).\end{equation}

\begin{lem}\label{lemmazo}
Assume that $a<c<2a$ and $c<b$. Let $Q\leq T$, let $(\gamma_{\bfn})_{\bfn}$ be any sequence of real numbers such that $\gamma_{\bfn}\in \mathcal{S}_{\bfn}$. Suppose that $(H_{\bfn}(t))_{\bfn}$ is a collection of functions for which \begin{equation}\label{G4G4G4}H_{\bfn}(t)=G_{\bfn}'(t)+O(t^{-1/2}\log t)\end{equation} for $t\in [Q/2,Q]$, the above implicit constant not depending on $\bfn$. Then one has that
\begin{equation*}\label{eqcU}\sum_{\tau_{n}\ll Q}P_{\bfn}^{-1/2}\min(\lvert H_{\bfn}(\gamma_{\bfn})\rvert^{-1},Q^{1/2})\ll Q^{3/4}(\log Q)^{3}+Q^{-1/2+(2a-c)/2(b-c)}.\end{equation*} 
\end{lem}

\begin{proof}
We shall denote henceforth for convenience by $W(Q)$ to the left side of the above equation.
The reader may find it useful to note that then
\begin{equation}\label{G22'}G_{\bfn}'(t)=a\log n_{1}-b\log n_{2}-c\log n_{3}+(b+c-a)\log t+\log \Big(\frac{b^{b}c^{c}}{a^{a}2^{b+c-a}}\Big).\end{equation} 
The evaluation of the above function at the point $\tau_{\bfn}$ shall play a not insignificant role in the course of the investigation cognate to this lemma. We thus recall (\ref{tautau}) and compute such an evaluation beforehand, say \begin{align}\label{Gtau23}3G_{\bfn}'(\tau_{\bfn})=&
(2b+2c+a)\log n_{1}+(2c-b-2a)\log n_{2}+(2b-c-2a)\log n_{3}+\log  K_{\bfa},\end{align} wherein $\log K_{\bfa}$ is a constant only depending on $\bfa$.
We note upon recalling (\ref{G4G4G4}) and (\ref{G22'}) in conjunction with the fact that $\gamma_{\bfn}\in\mathcal{S}_{\bfn}$ that \begin{equation}\label{gtaugtau23}H_{\bfn}(\gamma_{\bfn})=G_{\bfn}'(\tau_{\bfn})+O(Q^{-1/2}\log Q), \end{equation}
the above implicit constant being independent of $\bfn$. It also may be worth observing that in view of the assumptions on $a,b,c$ earlier made in the statement of the lemma then $2c< b+2a$. We thus introduce, for fixed $(n_{1},n_{3})$, the parameter \begin{equation}\label{N2N245}N_{2}=\big(K_{\bfa}n_{1}^{2b+2c+a}n_{3}^{2b-c-2a}\big)^{1/(b+2a-2c)}.\end{equation}
It has also been thought appropiate to define, for each triple $(n_{1},n_{2},n_{3})$ with $\bfn_{1}=(n_{1},n_{3})$ the number $r=n_{2}-N_{2},$ which may not be an integer, and write for ease of notation $H_{\bfn_{1},r}(t)$, $G_{\bfn_{1},r}(t),$ $\gamma_{\bfn_{1},r}$ and $\tau_{\bfn_{1},r}$ to denote $H_{\bfn}(t)$, $G_{\bfn}(t)$, $\gamma_{\bfn}$ and $\tau_{\bfn}$ respectively. By recalling (\ref{Gtau23}) it then transpires that 
\begin{align*}\label{Gau}3G_{\bfn_{1},r}'(\tau_{\bfn_{1},r})=&
(2b+2c+a)\log n_{1}+(2c-b-2a)\log (N_{2}+r)\nonumber
\\
&+(2b-c-2a)\log n_{3}+\log  K_{\bfa},\end{align*} whence utilising the fact that (\ref{Gtau23}) vanishes when substituting $n_{2}=N_{2}$ and combining it with (\ref{gtaugtau23}) one may deduce
$$H_{\bfn_{1},r}(t_{\bfn_{1},r})=\frac{2c-b-2a}{3}\log (1+r/N_{2})+O(Q^{-1/2}\log Q).$$

We denote as is customary by $\mathcal{G}_{1}$ to the set of integers $\lvert r\rvert\leq  N_{2}/2$ having the property that $\lvert H_{\bfn_{1},r}(\gamma_{\bfn_{1},r})\rvert\leq N_{2}^{-1}.$ In view of the uniformity in the above error term with respect to $r$, as was assumed in the statement of the lemma, it then transpires that \begin{equation*}\label{maria}\lvert \mathcal{G}_{1}\rvert\ll N_{2}Q^{-1/2}\log Q+1,\end{equation*} the contribution to $W(Q)$ stemming from the corresponding tuples being bounded above by
$$\sum_{n_{1}N_{2}n_{3}\ll Q^{3/2}}\sum_{r\in\mathcal{G}_{1}}n_{1}^{-1/2}N_{2}^{-1/2}n_{3}^{-1/2}\min\big(\lvert H_{\bfn_{1},r}(\gamma_{\bfn_{1},r})\rvert^{-1},Q^{1/2}\big)\ll W_{1}(Q)+W_{2}(Q),$$
wherein
\begin{equation*}\label{W1}W_{1}(Q)=(\log Q)\sum_{n_{1}N_{2}n_{3}\ll Q^{3/2}}n_{1}^{-1/2}N_{2}^{1/2}n_{3}^{-1/2}\end{equation*} and \begin{equation*}\label{W2}W_{2}(Q)=Q^{1/2}\sum_{n_{1}N_{2}n_{3}\ll Q^{3/2}}n_{1}^{-1/2}N_{2}^{-1/2}n_{3}^{-1/2}.\end{equation*} As a prelude to our analysis we note that the tuples involved in the above sums satisfy \begin{equation}\label{trescuartos}n_{1}^{1/2}N_{2}^{1/2}n_{3}^{1/2}\ll Q^{3/4}.\end{equation} We utilise such an estimate to obtain
\begin{equation}\label{pos}W_{1}(Q)\ll Q^{3/4}(\log Q)\sum_{n_{1}N_{2}n_{3}\ll Q^{3/2}}n_{1}^{-1}n_{3}^{-1}\ll Q^{3/4}(\log Q)^{3}.\end{equation}
In order to bound $W_{2}(Q)$ we define first, for convenience, the exponents 
$$\alpha_{1}=\frac{3b+3a}{b+2a-2c},\ \ \ \ \ \ \alpha_{3}=\frac{3b-3c}{b+2a-2c},$$
we remind the reader of (\ref{N2N245}) and observe that in view of the assumptions in the coefficient then $\alpha_{1}>\alpha_{3}$, and hence
\begin{align*}\label{pos2}W_{2}(Q)&
\ll Q^{1/2}\sum_{n_{1}^{\alpha_{1}}n_{3}^{\alpha_{3}}\ll Q^{3/2}}n_{1}^{-\alpha_{1}}n_{3}^{-\alpha_{3}}\ll Q^{-1/2+(2a-c)/2(b-c)}.\end{align*} 

It thus remains to analyse the contribution of the set $\mathcal{G}_{2}$ comprising integers $\lvert r\rvert\leq  N_{2}/2$ having the property that $\lvert H_{\bfn_{1},r}(\gamma_{\bfn_{1},r})\lvert> N_{2}^{-1}.$ Under such circumstances, it transpires that
\begin{align*}&\sum_{n_{1}N_{2}n_{3}\ll Q^{3/2}}
\sum_{r\in\mathcal{G}_{2}}n_{1}^{-1/2}N_{2}^{-1/2}n_{3}^{-1/2}\min\big(\lvert H_{\bfn_{1},r}(\gamma_{\bfn_{1},r})\rvert^{-1},Q^{1/2}\big)\ll W_{1}(Q),
\end{align*}
whence the application of (\ref{pos}) then completes the proof.

\end{proof}

We are now equipped to expeditiously analyse $B^{\psi}(Q)$ defined in (\ref{priuspsi}).
\begin{lem}\label{lemita615}
Assume that $a<c<2a$ and $c<b$. Then whenever $Q\leq T$ one has that
$$B^{\psi}(Q)\ll Q^{3/4}(\log Q)^{4}+Q^{-1/2+(2a-c)/2(b-c)}\log Q.$$ 
\end{lem}
\begin{proof}
We find it convenient to prepare the ground for our analysis by writing
\begin{equation*}G_{3}(t,y)=e^{(\log A)+1/t-I_{\bfa}y/2t-y^{2}/t}(1+iy)^{-1},\end{equation*} it being convenient to note for further purposes that such a function satisfies
\begin{equation}\label{g3G3G345}G_{3}(t,y)\ll (1+\lvert y\rvert)^{-1}.\end{equation}
We also introduce for $\bfn$ the corresponding phase function \begin{equation}\label{F1Fn23}F_{1,\bfn}(t,y)=y(\log A)+2y/t+I_{\bfa}/4t-I_{\bfa}y^{2}/4t+G_{\bfn}(t),\end{equation} wherein $G_{\bfn}(t)$ was defined in (\ref{GGGG}). In view of the decay exhibited by the function $G_{3}(t,y)$ in conjunction with (\ref{Kmuv}) and (\ref{priuspsi}) it then transpires that 
$$I_{\mathcal{S}_{\bfn}}^{\psi}(Q)=\int_{\mathcal{S}_{\bfn}}\int_{-Q^{1/2}\log Q}^{Q^{1/2}\log Q}G_{3}(t,y)e^{iF_{1,\bfn}(t,y)}dydt+O(Q^{-2}).$$ 

We focus on tuples satisfying $\mathcal{S}_{\bfn}\neq \o$, and hence $\tau_{\bfn}\ll Q$, since otherwise $I_{\mathcal{S}_{\bfn}}^{\psi}(Q)=0$. It seems worth noting for prompt use that an analogous argument to that utilised in Lemma \ref{lemita6155} enables one to assure that the derivative of $G_{3}(t,y)/F_{1,\bfn}'(t,y)$ with respect to $t$ vanishes in at most $O(1)$ points. We also find it desirable to recall (\ref{fundaA}) to the end of noting that whenever $t\in\mathcal{S}_{\bfn}$, as is the case herein, one has that $\lvert \log A\rvert\ll Q^{-1/2}\log Q.$ It then seems worth recalling (\ref{F1Fn23}) and observing that if $\lvert y\rvert\leq Q^{1/2}(\log Q)$ one has \begin{equation*}F_{1,\bfn}'(t,y)=G_{\bfn}'(t)+O\big(t^{-1/2}(\log t)\big),\end{equation*}  the corresponding implicit constant not depending on $\bfn$. The reader may notice that we have merely prepared the ground for an application of Lemma \ref{lemmazo}, it being convenient to denote by $s_{\bfn}$ to the real number $s\in\mathcal{S}_{\bfn}$ having the property that $\lvert F_{1,\bfn}'(s)\rvert$ is minimum in $\mathcal{S}_{\bfn}$, the existence of such a number being assured by the compactness of the set $\mathcal{S}_{\bfn}$. Therefore, combining \cite[Lemmata 4.3, 4.5]{Tit} with (\ref{g3G3G345}) and Lemma \ref{lemmazo} for the choice $H_{\bfn}(t)=F_{1,\bfn}'(t)$ one may deduce that
\begin{align*}B^{\psi}(Q)&\ll (\log Q) \sum_{\tau_{n}\ll Q}P_{\bfn}^{-1/2}\min(\lvert H_{\bfn}(s_{\bfn})\rvert^{-1},Q^{1/2})
\\
&\ll Q^{3/4}(\log Q)^{4}+Q^{-1/2+(2a-c)/2(b-c)}\log Q.
\end{align*}

\end{proof}

\section{An application of the stationary phase method}\label{sec6.7}
The remainder of the discussion shall be devoted to the analysis of $\tilde{B}^{\psi}(Q)$ defined in (\ref{priuspsi}), an application of the stationary phase method being required in due course. For such purposes we first apply Lemma \ref{lem0.2} to obtain
\begin{align*}\label{JS}I_{\tilde{\mathcal{S}}_{\bfn}}^{\psi}(Q)&=e^{i\pi \xi_{\bfa}/4}\int_{\tilde{\mathcal{S}}_{\bfn}\cap [\tau_{\bfn},Q]}e^{iG_{\bfn}(t)}dt+O(Q^{-2}),\end{align*} wherein $G_{\bfn}(t)$ was defined in (\ref{GGGG}) and $\tau_{\bfn}\leq Q$. It thus seems worth recalling (\ref{G22'}) and recording for further use that when writing \begin{equation}\label{cncn}c_{\bfn}=\Big(\frac{n_{2}^{b}n_{3}^{c}}{n_{1}^{a}}\Big)^{1/(b+c-a)}\eta_{\bfa},\ \ \ \ \ \ \ \text{with}\ \ \  \eta_{\bfa}=2\Big(\frac{a^{a}}{b^{b}c^{c}}\Big)^{1/(b+c-a)},\ \ \ \kappa=\eta_{\bfa}/2\pi,\end{equation} one then has $G_{\bfn}'(c_{\bfn})=0.$ We also find it desirable to note upon recalling (\ref{xi}) that $\xi_{\bfa}=-1$ in this context and
\begin{equation*}\label{G20}G_{\bfn}(c_{\bfn})=-(b+c-a)c_{\bfn}.\end{equation*} We shall provide an asymptotic evaluation of the term $\tilde{B}^{\psi}(Q)$, but before embarking in such an endeavour it seems desirable to denote $$Q_{\bfn}=\max(Q/2,\tau_{\bfn}),$$ and to write $\tilde{s}_{\bfn}$ to the real number $s\in\mathcal{S}_{\bfn}$ having the property that $\lvert G_{\bfn}'(s)\rvert$ is minimum in $\mathcal{S}_{\bfn}$. We then observe that an application of Titchmarsh \cite[Lemmata 4.2,4.4]{Tit} enables one to derive \begin{equation}\label{perif}\int_{\tilde{\mathcal{S}}_{\bfn}\cap [\tau_{\bfn},Q]}e^{iG_{\bfn}(t)}dt=\int_{Q_{\bfn}}^{Q}e^{iG_{\bfn}(t)}dt+O\big(\min(\lvert G_{\bfn}'(\tilde{s}_{\bfn})\rvert^{-1},Q^{1/2})\big).\end{equation} It seems worth foreshadowing that in the upcoming lemma we shall employ Lemma \ref{lemmazo} to estimate when averaging over $\bfn$ the above error term, it being pertinent to denote \begin{equation}\label{ober}\Lambda_{a,b,c}=(b+c-a)^{-1/2}\sqrt{2\pi}\ \ \ \ \ \ \text{and}\ \ \ \ \ \ \ \ \mu(\bfn)=P_{\bfn}^{-1/2}c_{\bfn}^{1/2}e^{iG_{\bfn}(c_{\bfn})}.\end{equation}
\begin{lem}\label{lem7.1.11}
Assume that $a<c<2a$ and $c<b$. Then one has for every $Q\leq T$ that
\begin{align*}\tilde{B}^{\psi}(Q)
=&\Lambda_{a,b,c}\sum_{\substack{Q_{\bfn}\leq c_{\bfn}\leq Q}}\mu(\bfn)+O\big(Q^{3/4}(\log Q)^{3}+Q^{1/4+(2a-c)/2(b-c)}\big).
\end{align*} 
\end{lem}
\begin{proof}
We begin our discussion by summing equation (\ref{perif}) over tuples $\bfn$ and applying \cite[Lemmata 4.2, 4.4]{Tit} in conjunction with \cite[Lemma 3.4]{Grako}, Lemma \ref{lemmazo} and (\ref{Pn}) to obtain
\begin{align*}\tilde{B}^{\psi}(Q)=&
\Lambda_{a,b,c}\sum_{\substack{Q_{\bfn}\leq c_{\bfn}\leq Q}}\mu(\bfn)+O\big(Q^{3/4}(\log Q)^{3}+Q^{-1/2+(2a-c)/2(b-c)}+E_{1}(Q)+E_{2}(Q)\big),
\end{align*} wherein upon denoting $R_{1}=Q$ and $R_{2}=Q_{\bfn}$ then 
$$E_{m}(Q)=\sum_{\substack{R_{m}/2\leq c_{\bfn}\leq 2R_{m}\\ \tau_{\bfn}\leq Q}}P_{\bfn}^{-1/2}\min\big(\lvert G_{\bfn}'(R_{m})\rvert^{-1},R_{m}^{1/2}\big),\ \ \ \ \ \ \ m=1,2.$$
We refer the reader to \cite[Lemma 6.1]{Pli} for the exposition of the details about such an application in a similar context. We shall begin by analysing the first error term in the above formula, and thus write
$$E_{2}(Q)=Y_{1}(Q)+Y_{2}(Q),$$ wherein $Y_{1}(Q)$ denotes the sum $E_{2}(Q)$ with triples satisfying $Q/2\leq \tau_{\bfn}\leq Q$, and $Y_{2}(Q)$ denotes the sum $E_{2}(Q)$ with triples subject to the proviso $\tau_{\bfn}< Q/2$. It then seems worth noting that under the constraints imposed on the tuples cognate to $Y_{1}(Q)$ and on recalling (\ref{SNN}), one may infer that $Q_{\bfn}\in\mathcal{S}_{\bfn}$. We have therefore reached a position from which to apply Lemma \ref{lemmazo} for the choice $H_{\bfn}(t)=G_{\bfn}'(t)$, namely
\begin{align*}Y_{1}(Q)&\ll \sum_{\substack{\tau_{\bfn}\ll Q}}P_{\bfn}^{-1/2}\min\big(\lvert G_{\bfn}'(Q_{\bfn})\rvert^{-1},Q_{\bfn}^{1/2}\big)\ll Q^{3/4}(\log Q)^{3}+Q^{-1/2+(2a-c)/2(b-c)}.\end{align*} 

We shift our attention to $Y_{2}(Q)$, recall as was previously done (\ref{cncn}) and define for fixed $(n_{1},n_{3})$ and further convenience the parameter
\begin{equation}\label{N1N1N1}N_{Q}=(Q/2)^{(b+c-a)/b}\eta_{\bfa}^{-(b+c-a)/b}n_{1}^{a/b}n_{3}^{-c/b}.\end{equation} On introducing for each $n_{2}\in\mathbb{N}$ the real number $r=n_{2}-N_{Q}$, recalling to the reader of (\ref{G22'}) and using the above line, it transpires that then \begin{align}\label{Geugeu}G_{\bfn}'(Q/2)=&a\log n_{1}-b\log (N_{Q}+r)-c\log n_{3}+(b+c-a)\log (Q/2)\nonumber
\\
&+\log \Big(\frac{b^{b}c^{c}}{a^{a}2^{b+c-a}}\Big)=-b\log (1+r/N_{Q}),\end{align} where we used the fact in view of the definition (\ref{N1N1N1}) that the right side of equation (\ref{G22'}) for the choices $t=Q/2$ and $n_{2}=N_{Q}$ vanishes. We note after an insightful inspection of the constraints in the tuples pertaining to the sum involved in the definition of $Y_{2}(Q)$ that it is apparent that $Q_{\bfn}=Q/2$, the underlying inequality cognate to $c_{\bfn}$ thus being equivalent to $Q/4\leq c_{\bfn}\leq Q,$ which may in turn be rephrased by means of the bounds $$2^{-(b+c-a)/b}N_{Q}\leq n_{2}\leq 2^{(b+c-a)/b} N_{Q}.$$We denote for simplicity by $I_{N_{Q}}$ to the above interval. We shall discuss first the instances for which $\lvert n_{2}-N_{Q}\rvert>1$, and herein a simple application of (\ref{Geugeu}) already delivers 
\begin{align}\label{CN111}
\sum_{\substack{\lvert n_{2}-N_{Q}\rvert> 1\\ n_{2}\in I_{N_{Q}}}}\frac{n_{2}^{-1/2}}{\lvert G_{\bfn}'(Q/2)\rvert}\ll \sum_{0<r\leq  N_{Q}}\frac{N_{Q}^{1/2}}{r}\ll N_{Q}^{1/2}\log N_{Q}.
\end{align}
We use the trivial bound $\min\big(\lvert G_{\bfn}'(Q/2)\rvert^{-1},Q^{1/2}\big)\ll Q^{1/2}$ if $\lvert n_{2}-N_{Q}\rvert\leq 1$ and combine such an observation with the preceding discussion to obtain
\begin{equation*}
Y_{2}(Q)\ll Y_{2,1}(Q)+Y_{2,2}(Q),
\end{equation*}
where $$Y_{2,1}(Q)=(\log Q)\sum_{n_{1}N_{Q}n_{3}\ll Q^{3/2}}n_{1}^{-1/2}n_{3}^{-1/2}N_{Q}^{1/2}$$ and $$Y_{2,2}(Q)=Q^{1/2}\sum_{n_{1}N_{Q}n_{3}\ll Q^{3/2}}(n_{1}N_{Q}n_{3})^{-1/2}.$$

We estimate $Y_{2,1}(Q)$ as in (\ref{pos}) and thus derive $Y_{2,1}(Q)=O\big(Q^{3/4}(\log Q)^{3}\big).$ In order to bound $Y_{2,2}(Q)$ it seems pertinent instead to note that $b> 2c-2a$, and observe after recalling (\ref{N1N1N1}) that \begin{align*}Y_{2,2}(Q)&\ll Q^{(a-c)/2b}\sum_{n_{1}^{a+b}n_{3}^{b-c}\ll Q^{(b-2c+2a)/2}}n_{1}^{-1/2-a/2b}n_{3}^{-1/2+c/2b}.
\\
&\ll Q^{-1/4+(b-2c+2a)/2(a+b)}\sum_{n_{3}^{b-c}\ll Q^{(b-2c+2a)/2}}n_{3}^{-(b-c)/(a+b)}\ll Q^{1/4+(2a-c)/2(b-c)},
\end{align*}
as desired. The combination of the estimates obtained above in conjunction with the constraints in the statement of the lemma enables one to deduce
$$Y_{1}(Q)+Y_{2}(Q)\ll Q^{3/4}(\log Q)^{3}+Q^{1/4+(2a-c)/2(b-c)}.$$

The analysis of $E_{1}(Q)$ shall be completely identical to the one cognate to $Y_{2}(Q)$ earlier exposed, whence in the interest of curtailing our discussion it has been thought preferable to indicate that the proof of an analogous estimate for it would follow by replacing $Q$ by $Q/2$ in (\ref{N1N1N1}), (\ref{Geugeu}) and (\ref{CN111}), such an observation combined with the above conclusion thus completing the proof of the lemma at hand.
\end{proof}

We shall merely combine a few of the preceding results in the upcoming corollary, it being pertinent recalling (\ref{ober}) to such an end.
\begin{cor}\label{cor2}
Assume that $a<c<2a$ and $c<b$. Then, when $T$ is sufficiently large one has that
$$I_{2}(T)=\Lambda_{a,b,c}\sum_{\substack{\tau_{\bfn}\leq c_{\bfn}\leq T}}\mu(\bfn)+O\big(T^{3/4}(\log T)^{4}+T^{1/4+(2a-c)/2(b-c)}\big).$$
\end{cor}
\begin{proof}
We use Lemmata \ref{lemita615} and \ref{lem7.1.11} in conjunction with (\ref{J1J23}) to deduce when $Q\leq T$ that
\begin{align*}\sum_{\bfn\in\mathbb{N}^{3}}P_{\bfn}^{-1/2}\int_{Q/2}^{Q}\psi(t)L_{\bfa}(\bfn)^{-it}K_{2}(P_{\bfn},t)dt=&\Lambda_{a,b,c}\sum_{Q_{\bfn}\leq c_{\bfn}\leq Q}\mu(\bfn)
\\
&+O\big(Q^{3/4}(\log Q)^{4}+Q^{1/4+(2a-c)/2(b-c)}\big).
\end{align*}
Summing over dyadic intervals accordingly in the preceding expression permits one to derive upon recalling (\ref{J21J21}) that
$$J_{2,1}(T)=\Lambda_{a,b,c}\sum_{\substack{\tau_{\bfn}\leq c_{\bfn}\leq T}}\mu(\bfn)+O\big(T^{3/4}(\log T)^{4}+T^{1/4+(2a-c)/2(b-c)}\big).$$ The corollary follows by utilising both the above equation and (\ref{osh}) combined with an application of Lemma \ref{J2uv}.
\end{proof}

\emph{Proof of Theorem \ref{thm2.3}}. When $a<c<2a$ then it follows that 
$$3/4<5/4-c/4a,$$ and it is apparent that the inequalities $$1/4+\frac{3a}{2(a+c)}<5/4-c/4a,\ \ \ \ \ \ \ \ \ \ 1/2+a/2c<5/4-c/4a$$ are equivalent to the condition $c^{2}+2a^{2}<3ac$, the latter holding in the aforementioned range. We shall conclude the proof of the aforementioned theorem by recalling equations (\ref{sigma}), (\ref{M1T}) and (\ref{integr}) and observing that then \cite[(7.5)]{Pli} yields
$$I_{a,b,c}(T)=\sigma_{a,b,c}T+M_{1}(T)+J_{2,2}(T)+O(T^{5/4-c/4a}).$$
Likewise, the discussion held in the present work permits one to combine equations (\ref{ecuacionp}) and (\ref{I111I245}) with Lemma \ref{lemmaiin} and Corollary \ref{cor2} to further derive 
\begin{align*}I_{a,b,c}(T)=&\sigma_{a,b,c}T+\Lambda_{a,b,c}\sum_{\substack{\tau_{\bfn}\leq c_{\bfn}\leq Q}}\mu(\bfn)+I_{1,2}(T)+O\big(T^{3/4}(\log T)^{4}+T^{1/4+(2a-c)/2(b-c)}\big).
\end{align*}
The combination of the preceding equations in conjunction with an application of Proposition \ref{maldini} then enables one to deduce 
\begin{equation}\label{ggf}\sum_{\substack{\tau_{\bfn}\leq c_{\bfn}\leq T}}\mu(\bfn)\ll T^{5/4-c/4a}+T^{1/4+(2a-c)/2(b-c)},\end{equation} as desired.
\appendix
\section{Van der Corput's estimates and exponent pairs}\label{pssd}
The rest of the memoir shall be devoted to present three natural alternatives for bounding the weighted exponential sum in the left side of (\ref{ggf}). Exhibiting all the possible methods to estimate such a sum hardly being the purpose of this note, the ones presented herein shall deliver weaker estimates for triples with $a<c<2a$ and $c< b$ on the range (\ref{range}). We first start with an application of van der Corput's second derivative test.
\begin{lem}\label{lemtiwst}
Whenever $a<c<2a$ and $c< b$ one has 
$$\sum_{\substack{\tau_{\bfn}\leq c_{\bfn}\leq T}}\mu(\bfn)\ll T^{3/4+(2a-c)/2(b-c)}.$$
\end{lem}
\begin{proof}
We begin as customary by making a dyadic dissection and examine for $Q\leq T$ the analogous sums with the triples satisfying the additional constraint $Q/2\leq c_{\bfn}\leq Q$. It transpires that under such a restriction one has
\begin{equation}\label{capu}n_{2}\asymp Q^{(b+c-a)/b}n_{1}^{a/b}n_{3}^{-c/b}.\end{equation} We find it worth observing that the ensuing condition in conjunction with the inequality $\tau_{\bfn}\leq Q$ entails the restriction
\begin{equation}\label{capu2}n_{1}^{a+b}n_{3}^{b-c}\ll Q^{(b-2c+2a)/2},\end{equation} it being worth noting that $b> 2c-2a$. We sum first over $n_{2}$ and apply van der Corput's second derivative test \cite[Theorem 5.9]{Tit} to obtain
\begin{align}\label{muuu}\sum_{\substack{Q/2\leq c_{\bfn}\leq Q\\ \tau_{\bfn}\leq c_{\bfn}}}\mu(\bfn)\ll S_{1}(Q)+S_{2}(Q),
\end{align}
where $$S_{1}(Q)=Q\sum_{n_{1}^{a+b}n_{3}^{b-c}\ll Q^{(b-2c+2a)/2}}\Big(Q^{(b+c-a)/b}n_{1}^{a/b}n_{3}^{-c/b}\Big)^{-1/2}n_{1}^{-1/2}n_{3}^{-1/2}$$ and $$S_{2}(Q)=\sum_{n_{1}^{a+b}n_{3}^{b-c}\ll Q^{(b-2c+2a)/2}}n_{1}^{-1/2}n_{3}^{-1/2}\Big(Q^{(b+c-a)/b}n_{1}^{a/b}n_{3}^{-c/b}\Big)^{1/2}.$$ It is apparent that the tuples pertaining to the above sums satisfy an inequality in the same vein as in (\ref{trescuartos}), whence an analogous argument would then yield
$$S_{2}(Q)\ll Q^{3/4}\sum_{n_{1}^{a+b}n_{3}^{b-c}\ll Q^{(b-2c+2a)/2}}n_{1}^{-1}n_{3}^{-1}\ll Q^{3/4}(\log Q)^{2}.$$

For the investigation of $S_{1}(Q)$ we note by rearranging terms that
\begin{align*}S_{1}(Q)&=Q^{(b-c+a)/2b}\sum_{n_{1}^{a+b}n_{3}^{b-c}\ll Q^{(b-2c+2a)/2}}n_{1}^{-1/2-a/2b}n_{3}^{-1/2+c/2b}
\\
&\ll Q^{1/4+(b-2c+2a)/2(a+b)}\sum_{n_{3}^{b-c}\ll Q^{(b-2c+2a)/2}}n_{3}^{-(b-c)/(a+b)}\ll Q^{3/4+(2a-c)/2(b-c)},
\end{align*}
whence summing over dyadic intervals permits one to derive the desired result.
\end{proof}

The upcoming lemma addresses the question of estimating the preceding exponential sum in a similar manner, the technical input employed in due course having its reliance instead on the use of exponent pairs.

\begin{lem}\label{piss}
Whenever $a<c<2a$ and $c <b$ one has 
$$\sum_{\substack{\tau_{\bfn}\leq c_{\bfn}\leq T}}\mu(\bfn)\ll T^{19/21+(2a-c)/4(b-c)+\varepsilon}.$$
\end{lem}

\begin{proof}
We begin as above by making a dyadic dissection and examine for $Q\leq T$ the sums with triples for which $Q/2\leq c_{\bfn}\leq Q$, it then being apparent that such triples satisfy (\ref{capu}) and (\ref{capu2}). We observe upon recalling (\ref{range}) that
$$\frac{\partial}{\partial n_{2}}G_{\bfn}(c_{\bfn})\asymp c_{\bfn}n_{2}^{-1}.$$ We employ the fact that by \cite[Theorem 6]{Bou} then $(13/84+\varepsilon, 55/84+\varepsilon)$ is an exponent pair to show via partial summation that
$$\sum_{\substack{Q/2\leq c_{\bfn}\leq Q\\ \tau_{\bfn}\leq c_{\bfn}}}\mu(\bfn)\ll Q^{55/84+\varepsilon}\sum_{n_{1}^{a+b}n_{3}^{b-c}\ll Q^{(b-2c+2a)/2}}n_{1}^{-1/2}n_{3}^{-1/2}.$$ Summing then over $n_{3}$ and $n_{1}$ successively permits one to deduce
\begin{align*}\sum_{\substack{Q/2\leq c_{\bfn}\leq Q\\ \tau_{\bfn}\leq c_{\bfn}}}\mu(\bfn)&\ll Q^{55/84+(b-2c+2a)/4(b-c)+\varepsilon}\sum_{n_{1}\ll Q^{(b-2c+2a)/2(a+b)}}n_{1}^{-1/2-(a+b)/2(b-c)}
\\
&\ll Q^{19/21+(2a-c)/4(b-c)+\varepsilon}.\end{align*}The result then follows as is customary by summing over dyadic intervals.
\end{proof}

It seems profitable to compare the preceding bounds with those stemming from (\ref{ggf}), the range of interest considered herein being that described in (\ref{range}). We first note that
$$5/4-c/4a-3/4-(2a-c)/2(b-c)=\frac{2ab-cb+c^{2}-4a^{2}}{4a(b-c)}=\frac{-(2a-c)(c+2a-b)}{4a(b-c)},$$whence in view of the condition $c<2a$ and (\ref{range}) it is apparent that 
$$5/4-c/4a<3/4+(2a-c)/2(b-c).$$ Likewise, a straightforward computation reveals that the inequality $$5/4-c/4a<19/21+(2a-c)/4(b-c)$$ is equivalent to $$b(29a-21c)<42a^{2}+8ac-21c^{2}.$$ If $29a\geq 21c$ then $$b(29a-21c)\leq (2a+c)(29a-21c)=58a^{2}-13ac-21c^{2}<42a^{2}+8ac-21c^{2}$$ since $a<c$. If on the contrary $29a<21c$ then $$b(29a-21c)<(42a+34c)(29a-21c)/55.$$We find it worth considering for the purpose of progressing in the discussion the function $f(x)=1092x^{2}+336x-441$ and note that it satisfies in the interval $(1/2,1)$ the inequality $f(x)>0$, a consequence of which being that $$b(29a-21c)<(42a+34c)(29a-21c)/55<42a^{2}+8ac-21c^{2},$$ as desired. Moreover, it transpires that the inequality $$1/4+(2a-c)/2(b-c)<19/21+(2a-c)/4(b-c)$$ is equivalent to $42a+34c<55 b.$ The preceding remarks permit one to assure that the estimates derived in Lemmata \ref{lemtiwst} and \ref{piss} are weaker than those stemming from (\ref{ggf}) whenever (\ref{range}) holds.

The last lines of the present section shall be devoted to apply the estimates available for three-dimensional exponential sums with monomials. We recall (\ref{cncn}) and introduce for positive integers $N_{1},N_{2},N_{3}$ and fixed real numbers $0<a,b,c< 1$ the cube 
$$\mathcal{N}=\Big\{(n_{1},n_{2},n_{3})\in \mathbb{N}^{3}:\ \ \ \ N_{1}\leq n_{1}\leq 2N_{1},\ N_{2}\leq n_{2}\leq 2N_{2},\ N_{3}\leq n_{3}\leq 2N_{3}  \Big\},$$ the exponential sum
\begin{equation}\label{kks}S_{0}=\sum_{(n_{1},n_{2},n_{3})\in\mathcal{N}}e\big(\kappa n_{2}^{b}n_{3}^{c}n_{1}^{-a}\big)\end{equation} and the parameter $X=N_{2}^{b}N_{3}^{c}N_{1}^{-a}$. 
Providing robust estimates for sums of the above type is not by all means sufficient to bound the exponential sum defined in the left side of (\ref{muuu}) in view of the range of summation underlying such a sum. It has nonetheless been thought pertinent to include the previous discussion to convey that any argument employed to estimate the left side of (\ref{muuu}) by covering the aforementioned set of summation with dyadic parallelepipeds and utilising bounds for such sums would deliver weaker bounds than those stemming from Theorem \ref{thm2.3} when the parameters $a,b,c$ satisfy the corresponding restrictions. It then seems worth drawing our attention back to (\ref{kks}) and noting that it is a consequence of Robert and Sargos \cite[Theorem 1]{Rob} that 
$$S_{0}\ll (N_{1}N_{2}N_{3})^{1+\varepsilon}\Big(\big(XN_{2}^{-1}N_{3}^{-2}N_{1}^{-1}\big)^{1/4}+(N_{1}N_{2})^{-1/4}+N_{3}^{-1/2}+X^{-1/2}\Big).$$
Therefore, recalling (\ref{ober}) then succesive applications of summation by parts yield
\begin{align*}\sum_{(n_{1},n_{2},n_{3})\in\mathcal{N}}\mu(\bfn)\ll & X^{1/2}(N_{1}N_{2}N_{3})^{1/2+\varepsilon}\Big(\big(XN_{2}^{-1}N_{3}^{-2}N_{1}^{-1}\big)^{1/4}+(N_{1}N_{2})^{-1/4}
\\
&+N_{3}^{-1/2}+X^{-1/2}\Big),
\end{align*}
whence in particular and for the sake of simplicity, for triples having the property that both \begin{equation}\label{oviedo}X\asymp T\ \ \ \ \ \text{and}\ \ \ \ \ N_{1}N_{2}N_{3}\asymp T^{3/2}\end{equation} it transpires that
$$\sum_{(n_{1},n_{2},n_{3})\in\mathcal{N}}\mu(\bfn)\ll T^{\varepsilon}(T^{7/8}N_{3}^{1/4}+T^{5/4}N_{3}^{-1/2}+T^{9/8}N_{3}^{-1/4}),$$the right side of the above equation being $\Omega(T)$. We thus observe that the present approach fails to provide satisfactory estimates for the above exponential sums over dyadic parallelepipeds when the parameters $N_{1},N_{2},N_{3}$ satisfy (\ref{oviedo}). 

In what follows we will show for completeness that such triples of parameters do indeed exist. The reader may note that the restriction (\ref{oviedo}) is equivalent to $$N_{2}^{a+b}N_{3}^{a+c}\asymp T^{b+c+a/2}\ \ \ \ \text{and}\ \ \ \ \ N_{2}N_{3}\ll T^{3/2},$$ the parameters $N_{2},N_{3}$ being positive integers. Combining the above provisos we infer that the existence of such tuples  $(N_{2},N_{3})$ is granted by the condition
$$T^{(b-a-c/2)/(b-c)}\ll N_{2}\ll T^{3/2}+T^{(b+c+a/2)/(a+b)},$$wherein the reader may find it useful to note that
$$\frac{b-a-c/2}{b-c}=1-\frac{(2a-c)}{2(b-c)}<1,$$where we utilised the proviso $c<2a$ and, by the condition $a<c$ then
$$\frac{b+c+a/2}{a+b}>1.$$ The preceding discussion enables one to deduce the existence of triples satisfying (\ref{oviedo}).

Similar approaches involving instead the use of available estimates for two dimensional exponential sums (see for instance Graham and Kolesnik \cite[Chapter 7]{Grako}, Fouvry and Iwaniec \cite[Theorem 1]{Fouv} or Liu \cite[Theorem 1.1]{Liu}) may be employed to obtain similar conclusions and derive bounds that are weaker to those stemming from (\ref{ggf}) for the ranges herein. We have omitted presenting such manoeuvres in the interest of curtailing our discussion.

\end{document}